\documentclass[12pt]{amsart}
\usepackage{amssymb}
\usepackage{mathrsfs}
\usepackage{amsxtra}
\usepackage{graphics}
\usepackage{latexsym}
\usepackage{xcolor}
\usepackage{amsmath}
\usepackage{amssymb,amsthm,amsfonts}
\usepackage{mathtools}
\usepackage{amscd}
\usepackage{tikz-cd}
\usepackage[arrow, matrix, curve]{xy}
\usepackage{syntonly}
\usepackage{paralist}
\usepackage[plainpages=false,pdfpagelabels,
colorlinks=true,linkcolor=blue,
citecolor=blue,filecolor=blue,      
urlcolor=blue,hypertexnames=false]{hyperref}
\usepackage{cleveref}
\usepackage{quiver} 
\usepackage{stmaryrd} \SetSymbolFont{stmry}{bold}{U}{stmry}{m}{n}   
\title[]{Nonlinear Hodge correspondence in positive characteristic}

\author[Mao Sheng]{Mao Sheng}

\begin{document}
\theoremstyle{plain}
\newtheorem{thm}{Theorem}[section]
\newtheorem{theorem}[thm]{Theorem}
\newtheorem{lemma}[thm]{Lemma}
\newtheorem{corollary}[thm]{Corollary}
\newtheorem{proposition}[thm]{Proposition}
\newtheorem{addendum}[thm]{Addendum}
\newtheorem{variant}[thm]{Variant}
\theoremstyle{definition}
\newtheorem{lemma and definition}[thm]{Lemma and Definition}
\newtheorem{construction}[thm]{Construction}
\newtheorem{statement}[thm]{Statement}
\newtheorem{notations}[thm]{Notations}
\newtheorem{question}[thm]{Question}
\newtheorem{problem}[thm]{Problem}
\newtheorem{remark}[thm]{Remark}
\newtheorem{remarks}[thm]{Remarks}
\newtheorem{definition}[thm]{Definition}
\newtheorem{claim}[thm]{Claim}
\newtheorem{assumption}[thm]{Assumption}
\newtheorem{assumptions}[thm]{Assumptions}
\newtheorem{properties}[thm]{Properties}
\newtheorem{example}[thm]{Example}
\newtheorem{conjecture}[thm]{Conjecture}
\newtheorem{proposition and definition}[thm]{Proposition and Definition}
\numberwithin{equation}{thm}
\newcommand{\Spec}{\mathrm{Spec}}
\newcommand{\pP}{{\mathfrak p}}
\newcommand{\sA}{{\mathcal A}}
\newcommand{\sB}{{\mathcal B}}
\newcommand{\sC}{{\mathcal C}}
\newcommand{\sD}{{\mathcal D}}
\newcommand{\sE}{{\mathcal E}}
\newcommand{\sF}{{\mathcal F}}
\newcommand{\sG}{{\mathcal G}}
\newcommand{\sH}{{\mathcal H}}
\newcommand{\sI}{{\mathcal I}}
\newcommand{\sJ}{{\mathcal J}}
\newcommand{\sK}{{\mathcal K}}
\newcommand{\sL}{{\mathcal L}}
\newcommand{\sM}{{\mathcal M}}
\newcommand{\sN}{{\mathcal N}}
\newcommand{\sO}{{\mathcal O}}
\newcommand{\sP}{{\mathcal P}}
\newcommand{\sQ}{{\mathcal Q}}
\newcommand{\sR}{{\mathcal R}}
\newcommand{\sS}{{\mathcal S}}
\newcommand{\sT}{{\mathcal T}}
\newcommand{\sU}{{\mathcal U}}
\newcommand{\sV}{{\mathcal V}}
\newcommand{\sW}{{\mathcal W}}
\newcommand{\sX}{{\mathcal X}}
\newcommand{\sY}{{\mathcal Y}}
\newcommand{\sZ}{{\mathcal Z}}
\newcommand{\A}{{\mathbb A}}
\newcommand{\B}{{\mathbb B}}
\newcommand{\C}{{\mathbb C}}
\newcommand{\D}{{\mathbb D}}
\newcommand{\E}{{\mathbb E}}
\newcommand{\F}{{\mathbb F}}
\newcommand{\G}{{\mathbb G}}
\renewcommand{\H}{{\mathbb H}}
\newcommand{\I}{{\mathbb I}}
\newcommand{\J}{{\mathbb J}}
\renewcommand{\L}{{\mathbb L}}
\newcommand{\M}{{\mathbb M}}
\newcommand{\N}{{\mathbb N}}
\renewcommand{\P}{{\mathbb P}}
\newcommand{\Q}{{\mathbb Q}}
\newcommand{\Qbar}{\overline{\Q}}
\newcommand{\R}{{\mathbb R}}
\newcommand{\SSS}{{\mathbb S}}
\newcommand{\T}{{\mathbb T}}
\newcommand{\U}{{\mathbb U}}
\newcommand{\V}{{\mathbb V}}
\newcommand{\W}{{\mathbb W}}
\newcommand{\Z}{{\mathbb Z}}
\newcommand{\g}{{\gamma}}
\newcommand{\id}{{\rm id}}
\newcommand{\rk}{{\rm rank}}
\newcommand{\END}{{\mathbb E}{\rm nd}}
\newcommand{\End}{{\rm End}}
\newcommand{\Hom}{{\rm Hom}}
\newcommand{\Hg}{{\rm Hg}}
\newcommand{\tr}{{\rm tr}}
\newcommand{\Sl}{{\rm Sl}}
\newcommand{\GL}{{\rm Gl}}
\newcommand{\Cor}{{\rm Cor}}

\newcommand{\SO}{{\rm SO}}
\newcommand{\OO}{{\rm O}}
\newcommand{\SP}{{\rm SP}}
\newcommand{\Sp}{{\rm Sp}}
\newcommand{\UU}{{\rm U}}
\newcommand{\SU}{{\rm SU}}
\newcommand{\SL}{{\rm SL}}
\newcommand{\ra}{\rightarrow}
\newcommand{\la}{\leftarrow}
\newcommand{\Gal}{\mathrm{Gal}}
\newcommand{\Res}{\mathrm{Res}}
\newcommand{\Gl}{\mathrm{Gl}}
\newcommand{\Gr}{\mathrm{Gr}}
\newcommand{\Exp}{\mathrm{Exp}}
\newcommand{\Sym}{\mathrm{Sym}}
\newcommand{\Ann}{\mathrm{Ann}}
\newcommand{\GSp}{\mathrm{GSp}}
\newcommand{\Tr}{\mathrm{Tr}}
\newcommand{\HIG}{\mathrm{HIG}}
\newcommand{\MIC}{\mathrm{MIC}}
\newcommand{\NHIG}{\mathrm{NHIG}}
\newcommand{\NMIC}{\mathrm{NMIC}}
\newcommand{\FV}{\mathrm{FV}}
\newcommand{\HV}{\mathrm{HV}}
\newcommand{\Ext}{\mathrm{Ext}}
\newcommand{\bA}{\mathbf{A}}
\newcommand{\bK}{\mathbf{K}}
\newcommand{\bM}{\mathbf{M}} 
\newcommand{\bP}{\mathbf{P}}
\newcommand{\bC}{\mathbf{C}}
\newcommand{\NMF}{\mathrm{NMF}}
\newcommand{\sFV}{\mathrm{SFV}}
\newcommand{\sHV}{\mathrm{SHV}}
\newcommand{\Aut}{{\rm Aut}}
\newcommand{\Gm}{{\mathbb{G}_m}}

\begin{abstract}
In this article, we extend the nonabelian Hodge correspondence in positive characteristic to the nonlinear setting.  
\end{abstract}

\maketitle

\section{Introduction}
In our recent work \cite{FS}, we introduced the notion of a \emph{nonlinear Higgs bundle} (\cite[Definition 1.3]{FS}), which is abstracted from the properties of nonabelian Kodaira-Spencer maps. On the other hand, one should also be able to abstract the properties of nonabelian Gauss-Manin connections, and make a category of nonlinear flat bundles (see below). Surprisingly, we may equalize these two categories in positive characteristic, subject to suitable conditions. The equivalence we are going to present is intimately related to the nonabelian Hodge correspondence established by Ogus-Vologodsky \cite{OV}, in their fundamental work in nonabelian Hodge theory in positive characteristic. Because of that, it seems appropriate to call this equivalence a \emph{nonlinear Hodge correspondence} in positive characteristic. \\

Let us first introduce several notions. Let $T$ be a scheme, and let $f: X\to S$ to be a $T$-morphism. 
\begin{definition}\label{lambda connection}
For a $\lambda\in \Gamma(T,\sO_T)$, a $\lambda$-connection on $f$ is an $\sO_{X}$-linear morphism
$$
\nabla: f^*T_{S/T}\to T_{X/T}
$$
such that its composite with the natural morphism $\pi: T_{X/T}\to f^*T_{S/T}$ is equal to $\lambda\cdot id$. 
\end{definition}
Note that the composite map 
$$
\wedge^2 f^*T_{S/T}\stackrel{[\nabla,\nabla]}{\longrightarrow} T_{X/T}\to T_{X/T}/\lambda\cdot \textrm{im}(\nabla),\  a\wedge b\mapsto \overline{[\nabla(a),\nabla(b)]}
$$
is $\sO_X$-linear when $\lambda$ is a unit. Moreover, $\nabla$ induces a decomposition 
$$T_{X/T}=T_{X/S}\oplus \nabla(f^*T_{S/T}),$$ and hence a natural isomorphism 
$$T_{X/T}/\lambda\cdot \textrm{im}(\nabla)\cong T_{X/S}.$$ 
In this case, we shall call either $f^*\wedge^2T_{S/T}\to T_{X/S}$ or its adjoint $\wedge^2T_{S/T}\to f_*T_{X/S}$ the \emph{curvature} of $\nabla$. A $\lambda$-connection $\nabla$, where $\lambda$ is a unit, is said to be integrable, if its curvature is zero. Another interesting case is $\lambda=0$. As $\ker(\pi)=T_{X/S}$, we may write a 0-connection as
$$
\theta: f^*T_{S/T}\to T_{X/S}.
$$
We call it integrable, if the composite, which is $\sO_S$-linear and shall be called the curvature of $\theta$,
$$
\wedge^2T_{S/T}\stackrel{[\theta,\theta]}{\longrightarrow} f_*T_{X/S}, \ a\wedge b\mapsto [\theta(a),\theta(b)]
$$
is zero. \\

In this article, we consider only the situation where $S\to T$ is smooth. Under this assumption, the notion of $\lambda$-connection behaves well under base change: Let $T'\to T$ be a morphism of schemes. Let $f_{T'}: X_{T'}\to S_{T'}$ be the base change of $f$ via $T'\to T$, and let $\lambda'\in \Gamma(T',\sO_{T'})$ be the image of $\lambda$. Then the composite
$$
f_{T'}^*T_{S_{T'}/T'}=f^*T_{S/T}|_{X_{T'}}\stackrel{\nabla}{\to} T_{X/T}|_{X_{T'}}\to T_{X_{T'}/T'}
$$
is a $\lambda'$-connection on $f_{T'}$. Also, the theory of $\lambda$-connections is intimately related to the theory of foliations. Indeed, an integrable $1$-connection on $f$ is nothing but a \emph{transversal foliation} on $f$. Now let $k$ be a base ring. A $k$-variety is by definition an integral separated scheme of finite type over $k$. 
\begin{definition}\label{nonlinear flat bundle and Higgs bundle}
Let $T$ be the spectrum of $k$, and $S$ be a smooth $k$-variety. A transversal foliated $S$-variety is a pair $(f,\nabla)$, where $f: X\to S$ is a morphism of $k$-varieties and $\nabla$ is an integrable $1$-connection (or connection for short) on $f$. A Higgs $S$-variety is a pair $(f,\theta)$, where $f$ is as before and $\theta$ is an integrable $0$-connection (or Higgs field for short) on $f$. A nonlinear flat (resp. Higgs) bundle over $S$ is a transversal foliated (resp. Higgs) $S$-variety $(f,\nabla)$ (resp. $(f,\theta)$) with $f$ being smooth. For two pairs $(f: X\to S,\nabla_f)$ and $(g: Y\to S,\nabla_g)$ of $\lambda$-connections, we define a morphism $(f,\nabla_f)\to (g,\nabla_g)$ to be an $S$-morphism $h: X\to Y$ such that 
$h^*\nabla_g=h_*\circ \nabla_f$, where $h_*: T_{X/k}\to h^*T_{Y/k}$ is the tangent map of $h$. 
\end{definition}

\begin{example}\label{nonabelian Hodge moduli}
Assume $k=\C$. Let $X\to S$ be a smooth projective morphism, and $G$ a reductive algebraic group over $\C$. Let $f: M^{sm}_{dR}(X/S,G)\to S$ be the sublocus of the relative de Rham space parameterizing Zariski dense flat connections. Simpson (\cite[\S8]{Si94},\cite[\S7-8]{Si97}) constructed the nonabelian Gauss-Manin connection on $f$, which is a nonlinear flat bundle in above sense (\cite[Proposition 4.2]{FS}). On the other hand, one has also $g: M^{sm}_{Dol}(X/S,G)\to S$, the sublocus of the relative Dolbeault space parameterizing Zariski dense Higgs bundles. We showed that $g$, equipped with the nonabelian Kodaira-Spencer map, is a nonlinear Higgs bundle (\cite[Theorem 1.2]{FS}).  
\end{example}
For a fixed $k$-variety $S$, it is clear that transversal foliated $S$-varieties (resp. Higgs $S$-varieties) form a category, which shall be denoted by $\NMIC(S)$ (resp. $\NHIG(S)$). The category $\NMIC(S)$ (resp. (resp. $\NHIG(S)$)) contains the category of nonlinear flat (resp. Higgs) bundles over $S$ as full subcategory, which shall be denoted by $\NMIC^{s}(S)$ (resp. $\NHIG^s(S)$). In \S\ref{pseduo-smooth morphisms}, we show that the category $\MIC(S)$ (resp. $\HIG(S)$) of coherent $\sO_S$-modules with integrable connections (resp. coherent $\sO_S$-modules with Higgs fields) can be embedded into $\NMIC(S)$ (resp. $\NHIG(S)$). Objects in these subcategories are said to be \emph{linear}.\\

  In the following statements, we assume that $k$ is a perfect field of characteristic $p>0$. Let $\sigma: k\to k$ be the Frobenius automorphism of $k$. Our first result is a nonlinear analogue of the classical descent theorem.
\begin{theorem}[Theorem \ref{equivalence in the case of vanishing curvature}, Theorem \ref{thm:nonlinear Cartier descent in appendix}]\label{nonlinear Cartier descent}
Suppose $S$ is smooth over $k$. Then there is an explicit equivalence of categories between the category $\NMIC_0(S)$ of transversal foliated $S$-varieties with vanishing $p$-curvature and the category $\NHIG_0(S')$ of Higgs $S'$-varieties with zero Higgs field, which satisfies the following properties:
\begin{itemize}
    \item [(i)] The equivalence restricts to an equivalence between the full subcategories of linear objects is the classical Cartier correspondence; 
    \item [(ii)] It restricts to an equivalence between $\NMIC^s_{0}(S)$ and $\NHIG^s_{0}(S')$.
\end{itemize}
\end{theorem}
Next, we obtain a nonlinear analogue of the Ogus-Vologodsky correspondence when $S$ is local.

\begin{theorem}[Theorem \ref{local nonlinear Hodge correspondence}]\label{local nonlinear Hodge correspondence main result}
Assume additionally the pair $(S, F_{S/k})$ is $W_2(k)$-liftable. Set $S'=S\times_{k,\sigma}k$. Then a lifting of $(S, F_{S/k})$ induces an explicit equivalence of categories between the category $\NMIC_{nil}(S)$ of transversal foliated $S$-varieties with nilpotent $p$-curvature and the category $\NHIG_{nil}(S')$ of Higgs $S'$-varieties with nilpotent Higgs field, which satisfies the following properties:
\begin{itemize}
    \item [(i)] It coincides with the Ogus-Vologodsky correspondence over the full subcategories of linear objects;
    \item [(ii)] It restricts to an equivalence between $\NMIC^s_{nil}(S)$ and $\NHIG^s_{nil}(S')$.
\end{itemize}
\end{theorem}
Our method is to construct explicit deformations of connections by using Frobenius liftings, and then apply  the nonlinear Cartier descent theorem. We also obtain a partial generalization of Theorem \ref{local nonlinear Hodge correspondence} to a global base, via the technique of exponential twisting (\cite{LSZ15}, \cite{SSW}).
\begin{theorem}[Theorem \ref{global nonlinear Hodge correspondence}]\label{global nonlinear Hodge correspondence main result}
Assume only the smooth $k$-variety $S$ is $W_2(k)$-liftable. Let $G$ be a split reductive group over $k$ with $h_{G}<p$, where $h_G$ is the Coxeter number of $G$. Then a $W_2(k)$-lifting of $S$ induces an equivalence of categories between $\NMIC^G_{nil}(S)$ and $\NHIG^G_{nil}(S')$.  
\end{theorem}
\begin{remark}
For precise description of the categories $\NMIC^G_{nil}(S)$ and $\NHIG^G_{nil}(S')$, we refer our readers to \S5. We remark that the previous results can be generalized to a more general base scheme $T$ in char $p$. We leave this task to interested readers.
\end{remark}
One may wonder what kind of objects in these categories \emph{naturally} correspond to each other. To this purpose, we introduce the category of \emph{nonlinear Fontaine modules} (over truncated Witt rings), which has natural functors into both categories. They are called nonlinear Fontaine modules, because we equip certain transversal foliated $S$-varieties with filtration structures as well as Frobenius structures, in a way very much like Faltings did in the linear setting (\cite[II. d)]{F}). See \S6 for details. 
\begin{theorem}[Theorem \ref{global one-periodicity}]
Assumption as in Theorem \ref{global nonlinear Hodge correspondence main result}. Let $(f,\nabla,Fil,\varphi)$ be a geometric strict $p$-torsion Fontaine module over $S$ with symmetry group $G$. Let $(g,\theta)$ be the associated graded Higgs $S$-variety to $(f,\nabla,Fil)$. Then $\varphi$ induces a natural isomorphism
$$
C^{-1}\pi^*(g,-\theta)\cong (f,\nabla),
$$
where $C^{-1}$ is the equivalence in Theorem \ref{global nonlinear Hodge correspondence main result}
from $\NHIG^G_{nil}(S')$ to $\NMIC^G_{nil}(S)$, and $\pi: S'\to S$ is the natural projection. 
\end{theorem}
In our forthcoming work \cite{Sh}, we prove that a variant of Example \ref{nonabelian Hodge moduli} in positive characteristic is a geometric strict $p$-torsion Fontaine module. The explicit description of the equivalence as given in the proof of Theorem \ref{local nonlinear Hodge correspondence main result} is crucial in our approach.\\
 
{\bf Acknowledgements.} The author would like to thank Caucher Birkar for the provoking question about a nonlinear generalization of the work \cite{LSZ19}, which is the starting point of the whole thing. He would like to thank Paolo Casini for drawing his attention to the work of T. Ekedahl \cite{E}, which plays the pivotal role in the early stage of our treatment of the nonlinear Cartier descent. Comments from and discussions with Yixuan Fu, Yupeng Wang, Jianping Wang, Jinxing Xu, Zhaofeng Yu and Lei Zhang (USTC) at various stages of the work are greatly appreciated. The author would like to especially thank Jinxing Xu for his generosity-the collaboration in Appendix is largely due to his work, that is vital to generalize our original result which are only for smooth morphisms to the present result which work for arbitrary morphisms. Last but not least, the author would like to thank Guizhou Key Laboratory of Pattern Recognition and Intelligent System for its hospitality, where the final stage of the work was done. This work is partially supported by the Chinese Academy of Sciences Project for Young Scientists in Basic Research (Grant No. YSBR-032).

\section{Linear $\lambda$-connections on linear schemes}\label{pseduo-smooth morphisms}
Let $T$ be a noetherian scheme, and $S\to T$ be a \emph{smooth} morphism. For a morphism $f: X\to S$, there is the canonically associated exact sequence of sheaves of relative K\"ahler differentials (\cite[Proposition 8.11]{H})
\begin{eqnarray}\label{exact sequence of differential forms}
f^*\Omega_{S/T}\to \Omega_{X/T}\to \Omega_{X/S}\to 0.    
\end{eqnarray}
Set $T_{X/T}$ to the linear dual $\mathcal{H}om_{\sO_X}(\Omega_{X/T},\sO_X)$ and etc. We obtain the following exact sequence of $\sO_X$-modules:
\begin{eqnarray}\label{exact seqeunce of tangent sheaves}
0\to T_{X/S}\to T_{X/T}\to f^*T_{S/T}.    
\end{eqnarray}
We shall call $f$ \emph{pseudo-smooth} if the above sequence is also exact on the right. As $S\to T$ is smooth by assumption, $\Omega_{S/T}$ is locally free of finite rank (see e.g. \cite[Corollary 18.58]{GW}), and hence $T_{S/T}$ is locally free of finite rank too. Thus we obtain a locally split short exact sequence for a pseudo-smooth $T$-morphism $f$
\begin{eqnarray}\label{short exact sequence of tangent sheaves}
0\to T_{X/S}\to T_{X/T}\to f^*T_{S/T}\to 0.   
\end{eqnarray}
Pseudo-smooth smoothness does not imply smoothness. For example, the natural projection $X=X_0\times_TS\to S$, where $X_0\to T$ is non smooth, is pseudo-smooth but non smooth. However, we have the following 
\begin{proposition}
Notation and assumption as above. Then $f: X\to S$ is smooth if and only if $X\to T$ is smooth and $f$ is pseudo-smooth.    
\end{proposition}
\begin{proof}
Suppose $f$ is smooth. Then the morphism $X\to T$ is smooth since it is the composite of two smooth morphisms. Also, the sequence \ref{exact sequence of differential forms} is exact on the left and locally split (see e.g. \cite[Corollay 18.72]{GW}). Hence we obtain a short exact sequence of locally free sheaves:
$$
0\to f^*\Omega_{S/T}\to \Omega_{X/T}\to \Omega_{X/S}\to 0.
$$
Taking the dual, we obtain the short exact sequence \ref{short exact sequence of tangent sheaves}. So $f$ is pseudo-smooth. Conversely, if $f$ is pseudo-smooth and $X\to T$ is smooth, we immediately see that $T_{X/S}$ is locally free as well. Taking the dual of the short exact sequence \ref{short exact sequence of tangent sheaves}, we get that the sequence \ref{exact sequence of differential forms} is exact on the left and locally split. Thus $f$ is smooth (see e.g. \cite[Corollay 18.72]{GW}).  
\end{proof}

Let $\sV$ be a coherent sheaf on $S$. Recall that a $\lambda$-connection on $\sV$, where $\lambda\in \Gamma(T,\sO_T)$, is an $\sO_T$-linear map 
$$
\nabla: \sV\to \sV\otimes \Omega_{S/T}
$$ 
satisfying the $\lambda$-Leibnitz rule
$$
\nabla(fs)=\lambda s\otimes df+f\nabla(s),\quad f\in \sO_S, \ s\in \sV.
$$
The curvature $F_{\nabla}$ of $\nabla$ is defined to be the composite
$$
\sV\stackrel{\nabla}{\to}\sV\otimes \Omega_{S/T}\stackrel{\nabla^1}{\to}\sV\otimes \wedge^2\Omega_{S/T},
$$
where $\nabla^1$ is an $\sO_T$-linear map determined by 
$$
\nabla^1(s\otimes \omega):=\nabla(s)\wedge \omega+\lambda s\otimes d\omega, \quad s\in \sV,\ \omega\in \Omega_{S/T}.
$$
It is $\sO_S$-linear. We say $\nabla$ is integrable when $F_{\nabla}=0$. \\

Recall that one may canonically associate an $S$-scheme to $\sV$, the so-called \emph{linear scheme} construction (see e.g. \cite[\S3]{N}). It is defined to be
$$
\pi: V:=\textbf{Spec} S(\sV)\to S
$$
where $S(\sV)$ is the symmetric algebra of $\sV$ over $\sO_S$. Linear schemes generalize the notion of geometric vector bundles (see e.g. \cite[Ch II, Ex. 5.18]{H}), which are the special cases where $\sV$ is locally free.  
\begin{proposition}\label{pseduo-smooth morphisms result}
Suppose there exists a $\lambda$-connection on $\sV$ where $\lambda$ is a unit. Then the associated linear scheme $\pi: V\to S$ to $\sV$ is pseudo-smooth.     
\end{proposition}
\begin{proof}
Since $\lambda^{-1}\nabla$ is simply a connection, we may assume $\lambda=1$ and $\nabla$ is a connection on $\sV$. The problem is locally in $S$. Hence we may assume $T=\Spec(A)$ and $S=\Spec(B)$. Then $\sV$ is the sheafification of $M=\Gamma(S,\sV)$, which is a finitely generated $B$-module. Set $C=S_B(M)$ to be the symmetric algebra attached to $M$ over $B$. Since $B$ is smooth over $A$, we may further localize to guarantee the existence of a set of coordinates 
$\{s_1,\cdots,s_m\}\subset B$ on $\Spec(B)$ with 
$$
\Omega_{B/A}=B\{ds_1\}\oplus\cdots B\{ds_m\}.
$$
Let $\{\frac{\partial}{\partial s_i}\}$ be the dual basis of $Der_{A}(B)=Hom_B(\Omega_{B/A},A)$. We are going to construct a $C$-morphism
$$
Der_A(B)\otimes_BC\to Der_A(C)
$$
whose composite with the natural $C$-morphism $Der_A(C)\to Der_A(B)\otimes_BC$ is the identity. Note that $Der_A(B)\otimes_BC$ is a free $C$-module with basis $\{\frac{\partial}{\partial s_i}\}$. Any of its $C$-morphism to $Der_A(C)$ is specified by a collection of elements $\{D_i\}\subset Der_A(C)$, such that $D_i$ will be the image of $\frac{\partial}{\partial s_i}$. We claim that the following association 
$$
ds_j\mapsto \delta_{ij}, 1\leq j\leq m, \quad da\mapsto \nabla_{\frac{\partial}{\partial s_i}}(a), \ a\in M
$$
defines an element $D_i$ in $Der_A(C)$. Indeed, by taking a set of generators $\{a_1,\cdots,a_r\}$ of $M$ as $C$-module, $\Omega_{C/A}$ is generated by $\{ds_1,\cdots,ds_m,da_1,\cdots,da_r\}$. So it suffices to verify the above association is well-defined. Let $\sum_j f_ja_j$ (with $\{f_j\}\subset B$) be a zero relation in $M$. We need to show $d(\sum_jf_ja_j)$, which is zero in $\Omega_{C/A}$, shall be mapped to zero under the association. As we have
\begin{eqnarray*}
    d(\sum_jf_ja_j)&=& \sum_j df_ja_j+f_jda_j\\
    &=&\sum_j(\sum_l\frac{\partial f_j}{\partial s_l}ds_l)a_j+f_jda_j,
\end{eqnarray*}
it will be sent to the element $\sum_j\frac{\partial f_j}{\partial s_i}a_j+f_j\nabla_{\frac{\partial}{\partial s_i}}a_j$. By the Leibnitz rule, it holds that
$$
\sum_j\frac{\partial f_j}{\partial s_i}a_j+f_j\nabla_{\frac{\partial}{\partial s_i}}a_j=\sum_j\nabla_{\frac{\partial}{\partial s_i}}(f_ja_j),
$$
which by additivity of $\nabla$ equals $\nabla_{\frac{\partial}{\partial s_i}}(\sum_jf_ja_j)$ and hence is zero. The claim is proved. It is clear that the so-constructed morphism satisfies the splitting property as requested. This completes the proof.
\end{proof}
The previous proposition enables us to construct many more examples of pseudo-smooth but non smooth morphisms in positive characteristic: For simplicity, let us take $T$ to be spectrum of an algebraically closed field of char $p>0$. Then for any coherent sheaf $\sV$ on $S$, $F_S^*\sV$ admits the canonical connection. By faithfully flat descent, $F_S^*\sV$ is non locally free when $\sV$ is non locally free. It follows that the associated linear scheme which is just $V\times_{S,F_S}S\to S$ is non smooth, while it is pseduo-smooth by Proposition \ref{pseduo-smooth morphisms result}.

\begin{proposition}\label{flat/higgs bundles are examples of foliated/higgs varieties}
Let $\sV$ be a coherent sheaf on $S$, and $\nabla$ be a $\lambda$-connection on $\sV$. Then over the linear scheme $\pi: V\to S$ attached to $\sV$, there exists a canonically associated $\lambda$-connection $\nabla_{\pi}$. Moreover, when $\lambda$ is either a unit or zero, $\nabla$ is integrable if and only if $\nabla_{\pi}$ is integrable.    
\end{proposition}
\begin{proof}
First we assume $S$ small affine as in the proof of Proposition \ref{pseduo-smooth morphisms result}. We shall modify the construction of a $C$-morphism 
$$
\alpha: Der_A(B)\otimes C\to Der_A(C)
$$ as follows: For each $i$, one sends $\frac{\partial}{\partial s_i}$ to $D_i\in Der_A(C)$ which is defined by
$$
ds_j\mapsto \lambda\delta_{ij}, 1\leq j\leq m, \quad da\mapsto \nabla_{\frac{\partial}{\partial s_i}}(a), \ a\in M.
$$
By the same argument as above (replacing the Leibnitz rule with the $\lambda$-Leibnitz rule), we know this is well-defined. Denote the morphism by $\tilde \nabla$. Clearly, its composite with the morphism $Der_A(C)\to Der_A(B)\otimes C$ is $\lambda\cdot id$. Next, assuming $\lambda$ is either a unit or zero, we shall verify that $\nabla$ is integrable if and only if $\tilde \nabla$ is integrable. As $\{a_1,\cdots,a_r\}$ is a generating system of $M$ as $B$-module, we may express for each triple $(i,j)$ (not necessarily in a unique way) 
$$
\nabla_{\frac{\partial}{\partial s_i}}a_j=\sum_kb^i_{jk}a_k, \quad b^i_{jk}\in B.
$$
The curvature 
$$
F_{\nabla}: \wedge^2Der_A(B)\to End_{B}(M)
$$ 
is therefore given by sending $\frac{\partial}{\partial s_i}\wedge \frac{\partial}{\partial s_j}, i<j$ to
$$
a_k\mapsto \sum_l[\lambda(\frac{\partial b^j_{kl}}{\partial s_i}-\frac{\partial b^i_{kl}}{\partial s_j})+\sum_m(b^j_{km}b^i_{ml}-b^i_{km}b^j_{ml})]a_l
$$
On the other hand, we may also compute that the Lie bracket $[D_i,D_j]\in Der_A(C)$ sends each $ds_k$ to zero, and $da_k$ to the element which is exactly given by the right hand side of the above expression. Thus, when $F_{\nabla}=0$, it follows immediately that the curvature of $\tilde \nabla$ is zero. As proved in Propostion \ref{pseduo-smooth morphisms result}, when $\lambda$ is a unit, one has the decomposition
$$
Der_A(C)=\alpha(Der_A(B)\otimes_BC)\oplus Der_B(C). 
$$
Note that $[D_i,D_j]\in Der_B(C)$ for any $i,j$. Hence, the vanishing of the curvature of $\tilde \nabla$ (which means the quotient of $[D_i,D_j]$ modulo $\alpha(Der_A(B)\otimes_BC)$ is zero for each $i,j$) forces $[D_i,D_j]$ to be zero. Thus by the above identification, $F_{\nabla}$ is zero too. When $\lambda=0$, one sees directly from the identification that $F_{\nabla}=0$. Finally, it remains to show $\lambda$-connection exists for a global $S$. For that purpose, we cover $S$ with small affines $\{S_{\alpha}\}_{\alpha}$ and denote the local $\lambda$-connections obtained in the first step by
$$
\nabla_{\alpha}: \pi^*T_{S_{\alpha}/T}\to T_{V_{\alpha}/T},
$$
where $V_{\alpha}=\pi^{-1}(S_{\alpha})$. We shall show that $\{\nabla_{\alpha}\}_{\alpha}$ glue into a global morphism $\nabla_{\pi}: \pi^*T_{S/T}\to T_{X/T}$. Let $\{s'_1,\cdots,s'_m\}$ be a set of coordinates on $S_{\beta}$, and over $S_{\beta}$, 
$$
\nabla_{\beta}(\frac{\partial}{\partial s'_i})=D'_i,
$$
where $D'_i(ds'_j)=\lambda\delta_{ij}$ and 
$$
D'_i(da_j)=\nabla_{\frac{\partial}{\partial s'_i}}(a_j)=\sum_kb^{'i}_{jk}a_k.
$$

Then it boils down to show, over the overlap $S_{\alpha}\cap S_{\beta}$, $\nabla_{\alpha}=\nabla_{\beta}$, that is for any $i$
$$
\nabla_{\beta}(\frac{\partial}{\partial s_i})=D_i.
$$
But this is more than obvious, since everything involved in the verification is linear. We shall omit this detail, and hence conclude the proof. 
\end{proof}
For the special case where $T=\Spec(k)$ and $\lambda=0,1$, we therefore obtain faithful functors 
$$
G: \MIC(S)\to \NMIC(S),\quad G: \HIG(S)\to \NHIG(S),
$$
which restrict to $\MIC^{lf}(S)\to \NMIC^s(S)$ (resp. $\HIG^{lf}(S)\to \NHIG^s(S)$). Here the supscript lf in front of $\MIC$ or $\HIG$ denotes for the subcategory consisting of locally free objects. Note also that $\nabla_{\frac{\partial}{\partial s_i}}(a)$ in the above construction belongs to $M$, which are the set of degree one elements in $C=S_A(M)$. Hence it is reasonable to call the connection $\nabla_{\pi}$ on the linear scheme $\pi: V\to S$ linear.

\section{Nonlinear Cartier descent theorem}\label{sec:nonlinear_cartier_descent_theorem}
Let $k$ be a perfect field of char $p$. Recall that $S'=S\times_{k,\sigma}k$ is the base change of $S$ via the Frobenius automorphism $\sigma: k\to k$. Let $F_{S/k}: S\to S'$ be the relative Frobenius morphism. Recall (\cite[Theorem 5.1]{K}) that the Cartier descent, that sends a coherent $\sO_S$-module $\sV$ equipped with an integrable $k$-connection $\nabla$ to the sheaf $\sV^{\nabla=0}$ of flat sections, is an equivalence of the category $\MIC_0(S)$ of pairs $(\sV,\nabla)$ with vanishing $p$-curvature and the category $\HIG_0(S')$ of coherent $\sO_{X'}$-modules. The quasi-inverse functor is to equip $F_{S/k}^*\sE$ with the so-called canonical connection $\nabla_{can}$ for an $\sE\in \HIG(S')$. The equivalence preserves locally freeness. In this section, we aim to establish the nonlinear Cartier descent theorem for \emph{smooth} morphisms. In the appendix, we shall extend it to arbitrary morphisms. \\

Our tool is the theory of inseparable morphisms due to T. Ekedahl \cite{E}. We start with recalling the notion of $p$-curvature of a foliation in char $p$. Let $X$ be a non-singular variety over $k$. Let $\sF\subset T_{X/k}$ be a foliation, that is a saturated coherent subsheaf of $T_{X/k}$, which is closed under Lie bracket. $\sF$ is said to be smooth if $T_{X/k}/\sF$ is locally free. Let 
$$
[p]: T_{X/k}\to T_{X/k}, D\mapsto D^p=D\circ\cdots\circ D\ (p\  \textrm{times}).
$$
be the power $p$ map. It is semi-linear and hence induces a linear morphism $[p]: F_X^*T_{X/k}\to T_{X/k}$, where $F_X$ is the absolute Frobenius morhpism. The $p$-curvature of $\sF$ is defined to be the composite morphism
$$
\psi_{\sF}: \sF\to T_{X/k}\stackrel{[p]}{\to} T_{X/k} \to \frac{T_{X}}{\sF}.
$$
$\sF$ is said to be $p$-closed if $\psi_{\sF}=0$. Recall the following basic result due to T. Ekedahl (see also \cite{M}).
\begin{theorem}[Proposition 2.4 \cite{E}]\label{Ekedahl's thm}
Notation as above. Then there is a one-to-one correspondence between the set of $p$-closed smooth foliations on $X$ and the following set of inseparable morphisms of $k$-varieties
$$
X\stackrel{\pi}{\to} Y\stackrel{\pi'}{\to} X'
$$
with $F_{X/k}=\pi'\circ \pi$ and $Y$ non-singular. 
\end{theorem}
This correspondence may be viewed as an integral version of the inseparable Galois theory introduced by N. Jacobson: One considers the natural action of the tangent bundle on the structure sheaf by derivation:  
$$
T_{X}\times \sO_X\to \sO_X,\  (D,a)\mapsto Da. 
$$
Given a foliation $\sF\subset T_{X/k}$, one defines a chain of finite morphisms $X\to Y\to X'$ by the following inclusions of sheaves of $k$-algebras
$$
\sO_X\supset \sO_Y\supset \sO_{X'},
$$
where 
$$
\sO_Y=\Ann(\sF):=\{a\in \sO_X|\  Da=0, \ \textrm{for}\ \forall D\in \sF\}.
$$
If $\sF$ is $p$-closed, then the degree of $\pi: X\to Y$ is equal to $p^{\rk \sF}$. $Y$ is always normal, and it is smooth if and only if $\sF$ is smooth. Conversely, given a chain of finite morphisms as above, one defines 
$$
\sF=\Ann(Y):=\{D\in T_{X/k}|\ Da=0, \ \textrm{for}\ \forall a\in \sO_Y\}.
$$
One verifies 
$$
\Ann\circ \Ann(\sF)=\sF,\quad \Ann\circ \Ann(Y)=Y.
$$

Let $g: Y\to S'$ be a smooth morphism over $k$. Let $\pi: X:=Y\times_{S'}S\to Y$ be the natural projection appearing in the following Cartesian diagram:
\[
		\xymatrix{ X\ar[d]_{f}\ar[r]^{\pi} &Y\ar[d]^{g} \\
			S\ar[r]^{F_{S/k}} & S'.}
		\]
For brevity, we shall write $Y\times_{S'}S$ by $F^*_{S/k}Y$. Let $f': X'\to S'$ be the base change of $f: X\to S$ via $\sigma: k\to k$. Then there is a commutative diagram of relative Frobenii:

\[
		\xymatrix{ X\ar[d]_{f}\ar[r]^{F_{X/k}} &X'\ar[d]^{f'} \\
			S\ar[r]^{F_{S/k}} & S'.}
		\]

There exists a canonical transversal foliation on $f$ as shown by the next lemma.
\begin{lemma}\label{factorization lemma}
There is a unique finite morphism $\pi': Y\to X'$ such that $\pi'\circ \pi=F_{X/k}$ and $g=f'\circ \pi'$. Let $\sF_{can}$ be the $p$-closed smooth foliation on $X$ by Theorem \ref{Ekedahl's thm}. Then $\sF_{can}$ is transversal with respect to $f$.  
\end{lemma}
Let $\nabla_{can}: f^*T_{S/k}\to T_{X/k}$ be the associated connection to $\sF_{can}$: The composite $\alpha: \sF_{can}\subset T_{X/k}\to f^*T_{S/k}$ is an isomorphism. Set $\nabla_{can}=\iota\circ\alpha^{-1}$, where $\iota$ is the natural inclusion of $\sF_{can}$ into $T_{X/k}$. 
\begin{proof}
Uniqueness is clear: Over the generic point, the morphism $\pi'$ is determined by the field inclusion $k(X')\subset k(Y)\subset k(X)$. To construct $\pi'$, we consider the quotient of $Y$ by the $p$-closed smooth foliation $T_{Y/S'}\subset T_{Y/k}$. So by Theorem \ref{Ekedahl's thm}, one obtains an inseparable morphism $\pi': Y\to Z$ where $\sO_Z\subset \sO_Y$ is the $\sO_{S'}$-subalgebra annihilated by $T_{Y/S'}$ (in particular, $\pi'$ is an $S'$-morphism). Since $Z$ is non-singular and the degree of the composite of inseparable morphisms $\pi'\circ \pi$ equals $p^{\rk(T_{Y/S'})+\rk(T_{S/k})}=p^{\dim(X)}$, it follows from Theorem \ref{Ekedahl's thm} that $Z=X'$ and $\pi'\circ \pi=F_{X/k}$. Also, the structural morphism $X'=Z\to S'$ is easily seen to be $f'$, the base change of $f$ by $\sigma$. It remains to show that $\sF_{can}$ is transversal. As $\sF_{can}$ and $f^*T_{S/k}$ are both locally free, the composite 
$$\sF_{can}=\Ann(\sO_Y)\to T_{X/k}\to f^*T_{S/k}$$ is an isomorphism if and only if for any closed point $x\in X$, the induced morphism on the fiber $$\sF_{can}(x)\to f^*T_{S/k}(x)=T_{S/k}(f(x))$$ is an isomorphism. Since the problem is strictly local, we may simply assume that $g: Y\to S'$ is given by the projection $\A_k^{m}\times \A_k^{r}\to \A_k^m$. In this case, the verification step is trivial. 
\end{proof}
Conversely, we are given a smooth morphism $f: X\to S$, together with a $p$-closed transversal foliation $\sF\subset T_{X/k}$. Using Theorem \ref{Ekedahl's thm}, we obtain a chain of finite morphisms $X\stackrel{\pi}{\to}Y \stackrel{\pi'}{\to}X'$ with $\pi'\circ \pi=F_{X/k}$. Let $g: Y\to S'$ be the composite $f'\circ \pi'$. By the universal property of Cartesian product, the map $\pi: X\to Y$ factorizes as the composite of a canonical morphism
$$
\alpha: X\to F_{S/k}^*Y
$$
with the natural projection $F_{S/k}^*Y\to Y$. Note that the degree of $\pi$ equals $p^{\dim S}$, as the rank of $\sF$ is $\dim S$ by transversality. Since the degree of the natural projection $F_{S/k}^*Y\to Y$ is also $p^{\dim S}$, it follows that $\deg \alpha=1$ and hence $\alpha$ is an isomorphism. Therefore $X=F_{S/k}^*Y$. In other words, we obtain the following commutative diagram
\[
		\xymatrix{  X\ar[d]_{f}\ar[r]^{\pi} &Y \ar[d]^{g}\ar[r]^{\pi'}&X'\ar[dl]^{f'} \\
			 S\ar[r]^{F_{S/k}} & S',&}
		\]
so that the left square is Cartesian. Since $f$ is smooth and $F_{S/k}$ is faithfully flat, it follows that $g$ is smooth (\cite[Corollaire 17.7.3 (ii)]{EGA IV}). Clearly, the above two constructions are converse to each other. 

\begin{theorem}\label{equivalence in the case of vanishing curvature}
The above construction defines an equivalence of categories between $\NMIC^s_0(S)$ and $\NHIG^s_0(S')$.
\end{theorem}
\begin{proof}
Let $h: Y_1\to Y_2$ be a morphism in $\NHIG^s_0(S')$. Pulling it back via $F_{S/k}$, one obtains the $S$-morphism
$\tilde h:=F_{S/k}^*h: F_{S/k}^*Y_1\to F_{S/k}^*Y_2$. Set $X_i=F_{S/k}^*Y_i$ and $F_i\subset T_{X_i}$ the transversal foliation corresponding to $\nabla_{can}$. To see that the composite map 
$$
F_1\to T_{X_1}\stackrel{\tilde h_*}{\to} \tilde h^*T_{X_2}
$$ factors through $\tilde h^*F_2\to \tilde h^*T_{X_2}$, it suffices to verify that for any $D_1\in F_1$ and any $a_2\in \sO_{Y_2}$, $\tilde h_*(D_1)(a_2)=0$. But this is clear, since 
$$
\tilde h_*(D_1)(a_2)=D_1(\tilde h^*a_2)=0.
$$
Thus the Frobenius pullback gives rise to a functor 
$$
C^{-1}: \NHIG^s_0(S')\to \NMIC^s_0(S).
$$
Similarly, one verifies that the converse construction maps a morphism in $\NMIC^s_0(S)$ to a morphism in $\NHIG^s_0(S')$ and hence a quasi-inverse functor
$$
C: \NMIC^s_0(S)\to \NHIG^s_0(S').
$$
\end{proof}

Following the line of arguments in Proposition \ref{flat/higgs bundles are examples of foliated/higgs varieties}, we are going to show the functor $G$ introduced there restricts to a functor $\MIC^{lf}_0(S)\to \NMIC^s_0(S)$.
\begin{lemma}\label{functor G}
Let $\nabla: \sV\to \sV\otimes \Omega_{S}$ be a flat bundle over $S$ and $\sF\subset T_{V/k}$ be the associated transversal foliation on $\pi: V\to S$. Then $\psi_{\nabla}=0$ if and only if $\psi_{\sF}=0$.
\end{lemma}
\begin{proof}
The proof relies on a comparison of local expressions of $\psi_{\nabla}$ and $\psi_{\sF}$. We use the notations in Proposition \ref{flat/higgs bundles are examples of foliated/higgs varieties}. Since $\sV$ is locally free, we may take $\{a_1,\cdots,a_r\}$ to be a local basis, and identify the dual basis $\{a^*_i\}$ with a set of local basis for $T_{V/S}$. Let $\frac{\partial}{\partial s}=\frac{\partial}{\partial {s_i}}$ and let $B=(b_{jk})$ be the connection matrix $\nabla_{\frac{\partial}{\partial s}}$ with respect to $\{a_j\}$. That is 
$$
\nabla_{\frac{\partial}{\partial s}}(a_j)=\sum_kb_{jk}a_k.
$$
Then 
$$
\psi_{\nabla}(\frac{\partial}{\partial s}\otimes 1)=(\nabla_{\frac{\partial}{\partial s}})^p=\nabla_{\frac{\partial}{\partial s}}\circ\cdots \circ\nabla_{\frac{\partial}{\partial s}} \ (p\  \textrm{times}).
$$ 
On the other hand, the corresponding $\psi_{\sF}(\frac{\partial}{\partial s}\otimes 1)$ is the image modulo $\sF$ of $D:=\frac{\partial}{\partial s}\otimes 1+\sum_j(\nabla_{\frac{\partial}{\partial s}}a_j) a_j^*$ under the power $p$ map. Write for $1\leq i\leq p$,
$$
(\nabla_{\frac{\partial}{\partial s}})^i(a_j)=\sum_k(B_i)_{jk}a_k,
$$
where $B_1=B$. Then a simple induction shows that 
$$
D^i=(\frac{\partial}{\partial s})^i\otimes 1+\sum_{j,k}(B_i)_{jk}a_ka^*_j.
$$ 
However, as $(\frac{\partial}{\partial s})^p=0$, we get the expression 
$$
D^p=\sum_{j,k}(B_p)_{jk}a_ka^*_j\in T_{V/S}. 
$$
Clearly, $\psi_{\nabla}(\frac{\partial}{\partial s}\otimes 1)$ and $D^p$ encode exactly the same information. Since the composite $T_{V/S}\to T_{V/k}\to T_{V/k}/\sF$ is an isomorphism, it follows that $\psi_{\nabla}(\frac{\partial}{\partial s}\otimes 1)=0$ if and only if $D^p\equiv 0 \mod\ \sF$. Finally, both $\psi_\nabla$ and $\psi_{\sF}$ are $p$-linear (see Lemma \ref{basic property of p-curvature} below), hence $\psi_{\nabla}=0$ if and only if $\psi_{\sF}=0$.  
\end{proof}
So we may embed $\MIC^{lf}_0(S)$ by $G$ as a subcategory of $\NMIC^s_0(S)$.  
\begin{proposition}
Notation as above. Then there is a commutative diagram of functors:
\[
		\xymatrix{ \MIC^{lf}_0(S)\ar[r]^{C} \ar[d]_{G} &\HIG^{lf}_0(S')\ar[d]^{G} \\
			\NMIC^s_0(S)\ar[r]^{C} & \NHIG^s_0(S'),}
		\]
where the top functor is the Cartier descent. 
\end{proposition}
\begin{proof}
For a locally free $\sO_{S'}$-module $\sE$ of finite rank, we are going to show
$$
G\circ C^{-1}(\sE,0)=C^{-1}\circ G(\sE,0).
$$
Let $E$ the geometric vector bundle $g: E\to S'$ associated to $\sE$. Clearly, the geometric vector bundle associated to $F_{S/k}^*\sE$ over $S$ is naturally isomorphic to $f: V=F_{S/k}^*E\to S$, the base of $g$ via $F_{S/k}$. To show canonical connections from two constructions coincide, we shall express them locally. Over $V'\subset S'$, we take a local basis $\{f_1,\cdots,f_r\}$ of $\sE|_{V'}$. Then $\{e_j=f_j\otimes 1\}$ forms a local basis for $F_{S/k}^*\sE|_U$. The local expression of the canonical connection on $F_{S/k}^*\sE$ is determined by $\nabla_{can}(e_j)=0$ for any $j$. Therefore, the corresponding transversal foliation on $f$ is locally given by the linear span of
$\{\frac{\partial}{\partial{s_1}}\otimes 1,\cdots, \frac{\partial}{\partial{s_m}}\otimes 1\}$. On the other hand, the canonical transversal foliation on $f$ constructed in Proposition \ref{equivalence in the case of vanishing curvature} is characterized by  
$\Ann(\sO_{E})\subset T_{V/k}$. As $\sO_E\subset \sO_V=\sO_E\otimes_{\sO_{S'},F^*_{S/k}}\sO_S$, it is clear that $\Ann(\sO_{E})|_{U}$ is also given the linear span of $\{\frac{\partial}{\partial{s_1}}\otimes 1,\cdots, \frac{\partial}{\partial{s_m}}\otimes 1\}$. So they are the same connection on $f$.   
\end{proof}
Therefore the nonlinear Cartier descent theorem extends the classical Cartier descent theorem in a natural way. 
\begin{remark}
In \S\ref{Appendix}, we give a different treatment of the nonlinear Cartier descent theorem. Consequently, the smoothness condition will be removed. Interestingly, the roles of Ekedahl's correspondence and the nonlinear Cartier descent will be switched therein. 
\end{remark}

\section{Local nonlinear Hodge correspondence}
In this section, we make the following assumption.
\begin{assumption}\label{assumption of scheme and frobenius lifting}
We say the pair $(S, F_{S/k})$ is $W_2(k)$-liftable, if the following condition is satisfied: There exists a smooth $W_2(k)$ scheme $\tilde S$ whose mod $p$ reduction is isomorphic to $S$. Fix a $k$-isomorphism $\alpha: \tilde S\times_{W_2(k)}k\cong S$. Then there exists a $W_2(k)$-morphism of schemes $\tilde F: \tilde S\to \tilde S$ whose mod $p$ reduction is $F_{S/k}$ under the isomorphism induced by $\alpha$.     
\end{assumption}
When the above assumption is satisfied for $S$, we say briefly that $S$ is \emph{local}. Affines are local, but projective curves of genus $\geq 2$ are not local. For local $S$, Ogus-Vologodsky \cite{OV} establishes an equivalence between the category of coherent modules with integrable connections whose $p$-curvatures are nilpotent and the category of coherent Higgs modules whose Higgs fields are nilpotent. We aim to generalize this correspondence to the nonlinear setting. In the following discussions, we fix an $W_2(k)$-lifting $\tilde S$ of $S$ and a lifting of $F=F_{S/k}$
$$
\tilde F: \tilde S\to \tilde S':=\tilde S\times_{W_2(k),W_2(\sigma)}W_2(k).
$$   

Let $(f,\nabla)$ be a transversal foliated $S$-variety. Set $\sF_{\nabla}=\textrm{im}(\nabla)\subset T_{X/k}$. As $\sF_{\nabla}$ splits the short exact sequence
$$
0\to T_{X/S}\to T_{X/k}\to f^*T_{S/k}\to 0,
$$
the $p$-curvature gives rise to a map 
$$
\psi_{\nabla}:=\psi_{\sF_{\nabla}}: f^*T_{S/k}\to T_{X}/\sF_{\nabla}\cong T_{X/S}.
$$
\begin{lemma}\label{basic property of p-curvature}
The following properties hold for $\psi_{\nabla}$:
\begin{itemize}
    \item [(i)] $\psi_\nabla$ is semi-linear;
    \item [(ii)] $\psi_\nabla$ is integrable;
    \item [(iii)] $[\psi(D_1),\nabla(D_2)]=0$ for any $D_1,D_2\in f^*T_{S/k}$.
\end{itemize}
\end{lemma}
By the lemma, the $p$-curvature of an integrable connection $\nabla$ on $f$ defines an $\sO_X$-linear morphism
$$
\psi_\nabla: F_X^*f^*T_{S/k}\to T_{X/S}, \quad [\psi_{\nabla},\psi_{\nabla}]=0.
$$
where $F_X$ is the absolute Frobenius of $X$. It is a nonlinear $F$-Higgs field on $f$. 
\begin{proof}
The first property holds for a general foliation on $X$. See e.g. \cite[Lemma 4.2 (ii)]{E}. Let $V\subset S$ be an open affine subset with a set of local coordinates $\{s_1,\cdots,s_m\}$. We need to show for any $D_1,D_2\in f^*T_{V/k}$, $[\psi_{\nabla}(D_1\otimes 1),\psi_{\nabla}(D_2\otimes 1)]=0$. By the $p$-linearity in (i), it suffices to show the vanishing for $D_1=\frac{\partial}{\partial{s_i}}, D_2=\frac{\partial}{\partial{s_j}}$ for any $i,j$. As $D_i^p=0$ in this case, it follows that 
$$
[\psi_{\nabla}(D_1),\psi_{\nabla}(D_2)]=[\nabla(D_1)^p,\nabla(D_2)^p].
$$
Now a simple calculation shows that $[\nabla(D_1),\nabla(D_2)]\in T_{X/S}$, and by the transversality, the composite 
$$
T_{X/S}\to T_{X/k}\to T_{X/k}/\sF_{\nabla}
$$
is an isomorphism (of Lie algebras). Hence the integrability of $\nabla$ implies that $[\nabla(D_1),\nabla(D_2)]=0$ (not only zero after modulo $F$). Hence $$[\nabla(D_1)^p,\nabla(D_2)^p]=0.$$  To show (iii), we write 
$$
D_1=\sum_i a_i\frac{\partial}{\partial{s_i}},\ D_2=\sum_j b_j\frac{\partial}{\partial{s_j}}.
$$
Then 
$$
\nabla(D_1)=\sum_ia_i\nabla(\frac{\partial}{\partial{s_i}}),\ \psi_{\nabla}(D_2)=\sum_jb_j^p\psi_{\nabla}(\frac{\partial}{\partial{s_j}}).
$$
Thus
$$
[\psi_{\nabla}(D_1),\nabla(D_2)]=\sum_{i,j}a_i^pb_j[\psi_{\nabla}(\frac{\partial}{\partial{s_j}}),\nabla(\frac{\partial}{\partial{s_i}})],
$$
and therefore it suffices to show (iii) for $D_1=\frac{\partial}{\partial{s_i}}, D_2=\frac{\partial}{\partial{s_j}}$ for any $i,j$. But this is clear as $[\nabla(D_1),\nabla(D_2)]=0$.
The lemma is proved.
\end{proof}
Note that $T_{X/S}$ is a restricted Lie subalgebra of $T_{X/k}$, it is meaningful to make the following 
\begin{definition}
Let $(f,\nabla)$ be a transversal foliated $S$-variety. Let $$\psi_{\nabla}: F_X^*f^*T_{S/k}\to T_{X/S}$$ be the associated $p$-curvature. It is said to be nilpotent if $[p]^n\circ\psi_{\nabla}=0$ for some non-negative integer $n$. It is said to be $p$-nilpotent if $[p]\circ\psi_{\nabla}=0$. 
\end{definition}
Replacing $\psi_{\nabla}$ in the above definition with the Higgs field $\theta$, we obtain the corresponding nilpotency conditions for a Higgs $S$-variety $(g,\theta)$. Denote $\NMIC_{nil}(S)$ (resp. $\NMIC^s_{nil}(S)$) for the category of transversal foliated $S$-varieties (resp. nonlinear flat bundles over $S$) with nilpotent $p$-curvature, and $\NHIG_{nil}(S')$ (resp. $\NHIG^s_{nil}(S')$) for the category of Higgs $S$-varieties (resp. nonlinear Higgs bundles over $S'$) with nilpotent Higgs fields.

\begin{theorem}\label{local nonlinear Hodge correspondence}
Notations as above. Then there is an explicit equivalence of categories between $\NMIC_{nil}(S)$ and $\NHIG_{nil}(S')$, which restricts to an equivalence between $\NMIC^s_{nil}(S)$ and $\NHIG^s_{nil}(S')$ and extends the correspondence in Theorem \ref{equivalence in the case of vanishing curvature}. 
\end{theorem}
\begin{remark}\label{local correspondence preserves flatness}
By Remark \ref{remark on flatness in nonlinear Cartier descent}, the foregoing proof actually also shows that the constructed equivalence preserves flatness of morphisms. That is, it sends objects $(f,\nabla)\in \NMIC_{nil}(S)$, where $f: X\to S$ is flat, to objects $(g: Y\to S',\theta)$ in $\NHIG_{nil}(S')$ with $g$ flat, and vice versa. 
\end{remark}
\begin{proof}
The whole proof is divided into several steps.\\

{\itshape Step 1.} The lifting of Frobenius $\tilde F$ induces a nonzero morphism
$$
\frac{d\tilde F}{[p]}: F^*\Omega_{S'}\to \Omega_S.
$$
Let $\zeta : T_{S}\to F_{S/k}^*T_{S'}$ be its linear dual. For a set of local coordinates $\{s_1,\cdots,s_m\}$ on $\tilde S$, $\tilde F$ is described by the set $\{a_i\}$ with $a_i\in \sO_{S}$, so that one may write
$$
\tilde F^*(s_i)=s_i^p+pa_i. \ 1\leq i\leq m.
$$
Then we have
$$
\zeta(\frac{\partial}{\partial{s_i}})=\sum_{j}f_{ij}\frac{\partial}{\partial{s_j}},
$$
where $f_{ij}=\frac{\partial a_j}{\partial{s_i}}+\delta_{ij}s_j^{p-1}$. The following lemma is obvious by the expression of $f_{ij}$.
\begin{lemma}\label{properties of zeta}
Notations as above. Then the following properties hold
\begin{itemize}
    \item [(i)] For any $i,j,k$,
    $$
    \frac{\partial f_{ik}}{\partial s_j}=\frac{\partial f_{jk}}{\partial s_i}.
    $$
    \item [(ii)] $\frac{\partial^{p-1}f_{ij}}{\partial{s_i}^{p-1}}=-\delta_{ij}$.
\end{itemize}
\end{lemma}

{\itshape Step 2.} Start with $(g,\theta)\in \NHIG_{nil}(S')$. We first do the construction in Theorem \ref{equivalence in the case of vanishing curvature} for $(g,0)$. So we obtain the pair $(f,\nabla_{can})$, where $f$ sits in the following Cartesian diagram
\begin{equation*}
		\begin{gathered}
		\xymatrix{
			X \ar[r]^{\pi} \ar[d]_{f} & Y \ar[d]^{g} \\
			S \ar[r]^{F_{S/k}} & S'.
		}
		\end{gathered}  
		\end{equation*}
Then the composite 
$$
\pi^*\theta\circ f^*\zeta: f^*T_{S/k}\stackrel{f^*\zeta}{\to} f^*F_{S/k}^*T_{S'}=\pi^*g^*T_{S'}\stackrel{\pi^*\theta}{\to} \pi^*T_{Y/S'}=T_{X/S}.
$$
is an element in $\Hom(f^*T_{S/k},T_{X/S})$. In above, the last equality (rather a canonical isomorphism) follows from $\pi^*\Omega_{Y/S'}=\Omega_{X/S}$ since this is a Cartesian diagram, and the fact that $\pi$ is flat. Note that the space of connections on $f$ is a torsor under $\Hom(f^*T_{S/k},T_{X/S})$. We are going to deform $\nabla_{can}$ into a new connection on $f$. More precisely, we define $\nabla: f^*T_{S/k}\to T_{X/k}$ by the formula
\begin{equation}\label{connection formula}
\nabla = \nabla_{can}-\sum_{n\geq 0}[p]^n\circ(\pi^*\theta\circ f^*\zeta)\circ (f^*\zeta)^n.    
\end{equation}
One should write $[p]F_X^*$ instead of $[p]$, if one insists in the linearity in the expression. Since $\theta$ is nilpotent, it is actually a finite sum, so the expression is always meaningful. In this step, we verify that $\nabla$ is integrable. The problem is local. As above, we take an open affine subset $V\subset S$ with a set of local coordinates $\{s_1,\cdots,s_m\}$. By base change via $\sigma$, we obtain an open subset $V'\subset S'$ with local coordinates $\{s'_1,\cdots,s'_m\}$. As $f$ is equipped with the canonical connection, it is pseudo-smooth. Let $U'\subset g^{-1}V'$ be an open affine subset with a set of generators $\{T'_1,\cdots,T'_r\}$ for the $\sO_{U'}$-module $T_{U'/V'}=T_{Y/S'}|_{U'}$. Then over $U=U'\times_{V'}V\subset f^{-1}(V)$, we may take a set of generators $\{T_1=T'_1\otimes 1,\cdots,T_r=T'_r\otimes 1\}$ for the $\sO_U$-module $T_{U/V}=T_{X/S}|_{U}$. Via the splitting induced by $\sF_{can}$, we obtain a set of generators  
$$\{T_1,\cdots,T_r,\frac{\partial}{\partial_{s_1}},\cdots,\frac{\partial}{\partial_{s_m}}\}$$
of the $\sO_U$-module $T_{U/k}=T_{X/k}|_{U}$. Under this choice, 
$$
\nabla_{can}(\frac{\partial}{\partial{s_i}})= \frac{\partial}{\partial{s_i}}.
$$  
Let $N$ be an integer such that $[p]^{N+1}\theta=0$. Set 
$$
\tilde \theta=\sum_{n=0}^{N}[p]^n\circ \pi^*\theta\circ (f^*\zeta)^n.
$$
Then one computes that 
$[\nabla(\frac{\partial}{\partial{s_i}}),\nabla(\frac{\partial}{\partial{s_j}}))]$ is equal to
$$
[\frac{\partial}{\partial{s_i}},-\tilde\theta\circ f^*\zeta(\frac{\partial}{\partial{s_j}})]+[\tilde \theta\circ f^*\zeta(\frac{\partial}{\partial{s_i}}),\frac{\partial}{\partial{s_j}}]+[\tilde\theta\circ f^*\zeta(\frac{\partial}{\partial{s_i}}),\tilde\theta\circ f^*\zeta(\frac{\partial}{\partial{s_j}})].
$$
The first two terms form a group, and the last term forms another. We shall show that they are both zero. Write $\tilde \theta(\frac{\partial}{\partial{s'_i}})=\sum_j\tilde\theta_{ij}T_j$ with $\tilde \theta_{ij}\in \sO_{Y}$.
We observe that for any $i,j,k$, 
$$
\frac{\partial \tilde \theta_{ij}}{\partial s_k}=0.
$$
To see this, let us look at the first two summands of $\tilde \theta$. Let us also write  $$\theta(\frac{\partial}{\partial{s'_i}})=\sum_j\theta_{ij}T'_j,$$ where $\theta_{ij}=\theta_{ij}(s'_1,\cdots,s'_m,t'_1,\cdots,t'_{r'})$ (where $\{t'_1,\cdots,t'_{r'}\}$ is a set of generators for the $\sO_{V'}$-algebra $\sO_{U'}$). Then
$$
\pi^*\theta(\frac{\partial}{\partial{s'_i}})=\sum_j\pi^*\theta_{ij}T_j,
$$
where $\pi^*\theta_{ij}=\theta_{ij}(s^p_1,\cdots,s_m^p,t'_1,\cdots,t'_{r'})$. Clearly, $\frac{\pi^*\theta_{ij}}{\partial s_k}=0$. We further write
$$
\pi^*\theta\circ f^*\zeta(\frac{\partial}{\partial{s'_i}})=\sum_{k,l}f_{ij}\pi^*\theta_{jk}T_k.
$$
It follows that 
$$
[p]\circ \pi^*\theta\circ f^*\zeta(\frac{\partial} {\partial{s'_i}})=\sum_j(\sum_{l}f^p_{il}\pi^*\theta^p_{lj})T_j.
$$
It is clear that 
$$
\frac{\partial (\sum_{l}f^p_{il}\pi^*\theta^p_{lj})}{\partial s_k}=0.
$$
Certainly, the same property holds for all $n\geq 1$ terms. With this observation in mind, it is quick to see that the first group equals
$$
\sum_{l,k}\tilde \theta_{lk}(\frac{\partial f_{il}}{\partial s_j}-\frac{\partial f_{jl}}{\partial s_i})T_k.
$$
Then Lemma \ref{properties of zeta} (i) allows us to conclude the first group is zero. Now we consider the second group. Although it would expand into a sum of many terms, it is zero for a very simple reason, namely $[\theta,\theta]=0$. Let us go to the detail. 
\begin{claim}\label{commutativity}
For any $n\geq 0, n'\geq 0$, and for any $D_1,D_2\in \pi^*g^*T_{S'/k}$,
$$
[[p]^n\circ \pi^*\theta (D_1),[p]^{n'}\circ \pi^*\theta (D_2)]=0.
$$    
\end{claim}
\begin{proof}
We do induction on $(n,n')$. Write 
$D_i=\sum a_{ij}\frac{\partial}{\partial_{s'_j}},\ a_{ij}\in \sO_X$. Consider $n=n'=0$. We compute that
$$
[\pi^*\theta(D_1),\pi^*\theta(D_2)]=\sum_{k,l}a_{1i}a_{2j}(\pi^*\theta_{il}T_l(\pi^*\theta_{jk})-\pi^*\theta_{jl}T_l(\pi^*\theta_{ik})T_k.
$$
It follows from $[\theta,\theta]=0$ that for any $i,j$,
$$
\sum_l [\theta_{il}T_l(\theta_{jk})-\theta_{jl}T_l(\theta_{ik})]=0.
$$
Thus $[\pi^*\theta(D_1),\pi^*\theta(D_2)]=0$. Consider $(n+1,n')$. We have
$$
[[p]^{n+1}\circ \pi^*\theta (D_1),[p]^{n'}\circ \pi^*\theta (D_2)]=ad([p]([p]^n\circ \pi^*\theta (D_1)))([p]^{n'}\circ \pi^*\theta (D_2)),
$$
which equals 
$$
ad([p]^n\circ \pi^*\theta (D_1))^p([p]^{n'}\circ \pi^*\theta (D_2)).
$$
It further equals 
$$
ad([p]^n\circ \pi^*\theta (D_1))^{p-1}(ad([p]^n\circ \pi^*\theta (D_1))([p]^{n'}\circ \pi^*\theta (D_2))).
$$
As 
$$
ad([p]^n\circ \pi^*\theta (D_1))([p]^{n'}\circ \pi^*\theta (D_2))=[[p]^n\circ \pi^*\theta (D_1),[p]^{n'}\circ \pi^*\theta (D_2)]
$$
which is zero by induction hypothesis, it follows that 
$$
ad([p]^n\circ \pi^*\theta (D_1))^p([p]^{n'}\circ \pi^*\theta (D_2))=0.
$$
\end{proof}
The above claim directly shows that the second group is also zero. We have therefore shown the integrability of $\nabla$. \\

{\itshape Step 3.}
In this step, we will establish the $p$-curvature formula for $\nabla$ constructed in step 2.  
\begin{lemma}\label{p-curvature formula}
Let $\nabla$ be the integrable connection in Formula \ref{connection formula}. Then it holds that
$$
\psi_{\nabla}=\pi^*\theta.
$$
\end{lemma}
\begin{proof}
To prove it, we use the Jacobson formula (see Formula \cite[(5.2.4), (5.2.5)]{K}). It suffices to show for any $D\in \{\frac{\partial}{\partial{s_i}}\}$, the above equality holds for $D$. In this case,
$$
\psi_{\nabla}(D)=\nabla(D)^p=(\nabla_{can}(D)-\tilde\theta\circ f^*\zeta(D))^p,
$$
which is equal to 
$$
\nabla_{can}(D)^p+(-1)^p(\tilde \theta\circ f^*\zeta(D))^p+\sum_{i=1}^{p-1}s_i(\nabla_{can}(D),-\tilde\theta\circ f^*\zeta(D)),
$$
where $s_i$s are the universal Lie polynomials. In above, the first term is simply zero. The second term is actually the evaluation at $D$ of
$$
-[p]\circ \tilde \theta\circ f^*\zeta=-\sum_{n=1}^{N}[p]^n\circ \pi^*\theta\circ (f^*\zeta)^n-[p]^{N+1}\circ \pi^*\theta\circ (f^*\zeta)^{N+1}.
$$
For $[p]^{N+1}\circ \theta=0$, the second term in the right hand side of the above expression vanishes. 
\begin{claim}\label{universal Lie polynomial}
$$
\sum_{i=1}^{p-1}s_i(\nabla_{can},-\tilde\theta\circ f^*\zeta)=\tilde \theta. 
$$
\end{claim}
\begin{proof}
We shall compute 
$$
\textrm{ad}(t\nabla_{can}(D)-\tilde \theta\circ f^*\zeta(D))^{p-1})(\nabla_{can}(D))
$$
for $D=\frac{\partial}{\partial s_i}$ explicitly. Compute 
$$
[t\nabla_{can}(D)-\tilde \theta\circ f^*\zeta(D),\nabla_{can}(D)]=-[\tilde \theta\circ f^*\zeta(D),\nabla_{can}(D)]=\sum_{j,k}\frac{\partial f_{ij}}{\partial s_i}\tilde \theta_{jk}T_k.
$$
For $p=2$, we may already conclude the claim using Lemma \ref{properties of zeta} (ii). For an odd $p$, we compute the next term 
\begin{eqnarray*}
[t\frac{\partial}{\partial s_i}-\sum_{j,k}f_{ij}\tilde \theta_{jk}T_k,\sum_{j,k}\frac{\partial f_{ij}}{\partial s_i}\tilde \theta_{jk}T_k]&=&t[\frac{\partial}{\partial s_i},\sum_{j,k}\frac{\partial f_{ij}}{\partial s_i}\tilde \theta_{jk}T_k]    \\
&-& [\sum_{j,k}f_{ij}\tilde \theta_{jk}T_k,\sum_{j,k}\frac{\partial f_{ij}}{\partial s_i}\tilde \theta_{jk}T_k]
\end{eqnarray*}
Clearly, the first term in the right hand side gives us $t\sum_{j,k}\frac{\partial^2 f_{ij}}{\partial s_i^2}\tilde \theta_{jk}T_k$. The upshot is that the second term vanishes: First of all, by Claim \ref{commutativity}, it holds that for any $D_1,D_2\in \pi^*g^*T_{S'}$, $[\tilde \theta(D_1),\tilde \theta(D_2)]=0$. By setting
$$
D_1=\sum_jf_{ij}\frac{\partial}{\partial s'_j},\ D_2=\sum_j\frac{\partial f_{ij}}{\partial s_i}\frac{\partial}{\partial s'_j},
$$
we obtain
$$
\tilde \theta(D_1)=\sum_{j,k}f_{ij}\tilde \theta_{jk}T_k,\ \tilde \theta(D_2)=\sum_{j,k}\frac{\partial f_{ij}}{\partial s_i}\tilde \theta_{jk}T_k.
$$
Hence the second term is indeed zero. Summarizing it, we have obtained 
$$
\textrm{ad}(t\nabla_{can}(\frac{\partial}{\partial s_i})-\tilde \theta\circ f^*\zeta(\frac{\partial}{\partial s_i}))^{2}(\nabla_{can}(\frac{\partial}{\partial s_i}))=t\sum_{j,k}\frac{\partial^2 f_{ij}}{\partial s_i^2}\tilde \theta_{jk}T_k.
$$
Repeating this process, we get that
$$
\textrm{ad}(t\nabla_{can}(\frac{\partial}{\partial s_i})-\tilde \theta\circ f^*\zeta(\frac{\partial}{\partial s_i}))^{p-1}(\nabla_{can}(\frac{\partial}{\partial s_i}))=t^{p-2}\sum_{j,k}\frac{\partial^{p-1} f_{ij}}{\partial s_i^{p-1}}\tilde \theta_{jk}T_k.
$$
Thus by Lemma \ref{properties of zeta} (ii), we get that
$$
\textrm{ad}(t\nabla_{can}(\frac{\partial}{\partial s_i})-\tilde \theta\circ f^*\zeta(\frac{\partial}{\partial s_i}))^{p-1}(\nabla_{can}(\frac{\partial}{\partial s_i}))=-t^{p-2}\tilde \theta(\frac{\partial}{\partial s'_i}).
$$
Therefore,
$$
s_i(\nabla_{can}(D),-\pi^*\theta\circ f^*\zeta(D))=0, \ 1\leq i\leq p-2,
$$
and 
$$
s_{p-1}(\nabla_{can}(D),-\tilde \theta\circ f^*\zeta(D))=(p-1)^{-1}(-\tilde \theta(D))=\tilde \theta(D)
$$
\end{proof}
The lemma follows from Claim \ref{universal Lie polynomial} and the discussion before the claim.
\end{proof}
Because of Lemma \ref{p-curvature formula}, $\psi_{\nabla}$ is nilpotent. Hence the pair $(f,\nabla)$ belongs to the category $\NMIC_{nil}(S)$.\\

{\itshape Step 4.} 
Conversely, for $(f,\nabla)\in \NMIC_{nil}(S)$, we define a new connection on $f$ by
\begin{equation}\label{new connection}
\nabla_{new}=\nabla+\sum_{n\geq 0}[p]^n\circ (\psi_{\nabla}\circ f^*\zeta)\circ (f^*\zeta)^n,  
\end{equation}
where $\psi_{\nabla}\circ f^*\zeta$ reads
$$
f^*T_{S/k}\stackrel{f^*\zeta}{\to} f^*F_{S/k}^*T_{S'}=F_{X/k}^*f'^*T_{S'}=F_X^*f^*T_{S/k} \stackrel{\psi_{\nabla}}{\to} T_{X/S}.
$$
Again, because of the nilpotency condition on $\psi_{\nabla}$, the above sum is actually a finite sum. We are going to show that $\nabla_{new}$ is integrable and $p$-closed. In the expansion of the Lie bracket $[\nabla_{new},\nabla_{new}]$, we have these forms
$$
\{ [\nabla,\nabla],\ [\nabla, [p]^n\circ (\psi_{\nabla}\circ f^*\zeta)\circ (f^*\zeta)^n],\ [[p]^n\circ (\psi_{\nabla}\circ f^*\zeta)\circ (f^*\zeta)^n, [p]^{n'}\circ (\psi_{\nabla}\circ f^*\zeta)\circ (f^*\zeta)^{n'}\}.
$$
The first one vanishes because of the integrability of $\nabla$. For $n=0$ in the second form, we apply directly Lemma \ref{basic property of p-curvature} (iii). For $n\geq 1$, we express it as
$$
-ad([p]^{n}\circ (\psi_{\nabla}\circ f^*\zeta)\circ (f^*\zeta)^n)(\nabla)=-ad([p]^{n-1}\circ (\psi_{\nabla}\circ f^*\zeta)\circ (f^*\zeta)^n)^p(\nabla),
$$
and then apply an induction on $n$. The vanishing of the third form follows from the argument in the proof of Claim \ref{commutativity}, replacing the integrablity of $\theta$ by Lemma \ref{basic property of p-curvature} (ii). These altogether show the integrability of $\nabla_{new}$.\\

To show the $p$-closedness of $\nabla_{new}$, we use the same argument as Lemma \ref{p-curvature formula}. Suppose $[p]^{N+1}\psi_{\nabla}=0$, and set $\tilde \psi_{\nabla}=\sum_{n=0}^N[p]^n\circ \psi_{\nabla} \circ (f^*\zeta)^n$. Then by the Jacobson formula, we get for $D=\frac{\partial}{\partial s_i}$,
$$
\psi_{\nabla_{new}}(D)=\psi_{\nabla}(D)+(\tilde \psi_{\nabla}\circ f^*\zeta(D))^p+\sum_{i=1}^{p-1}s_i(\nabla(D),\tilde \psi_{\nabla}\circ f^*\zeta(D)).
$$
The second term equals 
$$
\sum_{n=1}^{N+1}[p]^n\circ \psi_{\nabla} \circ (f^*\zeta)^n(D)=\sum_{n=1}^N[p]^n\circ \psi_{\nabla} \circ (f^*\zeta)^n(D).
$$
To get the description of the third term, we compute that
$$
\textrm{ad}(t\nabla(\frac{\partial}{\partial s_i})+\tilde \psi_{\nabla}\circ f^*\zeta(\frac{\partial}{\partial s_i}))^{p-1}(\nabla(\frac{\partial}{\partial s_i}))=\tilde \psi_{\nabla}(D)t^{p-2}.
$$
We omit the full detail, because the calculation is very similar to that of Lemma \ref{p-curvature formula}, except noting the following point: Write 
$$
\nabla(\frac{\partial}{\partial s_i})=\frac{\partial}{\partial s_i}+\sum_ja_{ij}T_j.
$$ 
Then it is clear that for any $a\in \sO_S$, 
$$
\nabla(\frac{\partial}{\partial s_i})(a)=\frac{\partial a}{\partial s_i}.
$$
So it looks as if $\nabla_{can}$ acts on it. 
It follows that 
$$
s_i(\nabla(D),\tilde \psi_{\nabla}\circ f^*\zeta(D))=0,\ 1\leq i\leq p-2,
$$
and 
$$
s_{p-1}(\nabla(D),\tilde \psi_{\nabla}\circ f^*\zeta(D)=-\tilde \psi_{\nabla}(D)=-\psi_{\nabla}(D)-\sum_{n=1}^N[p]^n\circ  \psi_{\nabla}\circ (f^*\zeta)^n
$$
Therefore $\psi_{\nabla_{new}}(D)=\psi_{\nabla}(D)-\psi_{\nabla}(D)=0$ as claimed.\\

{\itshape Step 5.} By Theorem \ref{equivalence in the case of vanishing curvature}, the pair $(f,\nabla_{new})$ descends to an object $(g,0)\in \NHIG_0(S')$, that is, there is a unique (up to isomorphism) morphism $g: Y\to S'$ such that $X=F_{S/k}^*Y$ as $S$-variety and $\nabla_{new}=\nabla_{can}$. So we may express the $p$-curvature of $\nabla$ as
$$
\psi_{\nabla}: \pi^*g^*T_{S'}=F_X^*f^*T_{S/k}\to T_{X/S}=\pi^*T_{Y/S'}.
$$
But since
$$
[\nabla_{can},\psi_{\nabla}]=[\nabla,\psi_{\nabla}]+[\sum_{n\geq 0}[p]^n\circ (\psi_{\nabla}\circ f^*\zeta)\circ (f^*\zeta)^n,\psi_{\nabla}]=0,
$$
it follows that $\psi_{\nabla}$ induces a morphism
$$
\theta: g^*T_{S'} \to T_{Y/S'},
$$
such that $\pi^*\theta=\psi_{\nabla}$. Clearly, $\theta$ is nilpotent, and hence the pair $(g,\theta)$ belongs to the category $\NHIG_{nil}(S')$. \\

{\itshape Step 6.} 
The constructions in Step 3 and Step 5 are converse to each other. Start with $(g: Y\to S,\theta)\in \NHIG_{nil}(S)$. Then we obtain $(f,\nabla)$ in Step 3 and then $(g',\theta')$ out of $(f,\nabla)$ in Step 5. By construction, $f: X\to S$ is the one sitting in the Cartesian diagram in Step 2 and $g': Y'\to S$ in the one by the nonlinear Cartier descent applied to $(f,\nabla_{can})$. Since the nonlinear Cartier descent is an equivalence of categories, $g'$ is canonically isomorphic to $g$. To see $\theta'$ coincides with $\theta$, it suffices to show $\psi_{\nabla}=\pi^*\theta$, because $\theta'$ is the unique descent from $\psi_{\nabla}$ by the nonlinear Cartier descent theorem. But this is exactly what Lemma \ref{p-curvature formula} says. The other direction is argued similarly. It is routine to verify that the constructions preserve morphisms in the categories. Therefore, we may obtain an equivalence of categories. Moreover, when $\theta=0$ in Step 3 and $\psi_{\nabla}=0$ in Step 5, the constructions clearly reduce back to the construction of the nonlinear Cartier descent. The statement that this equivalence restricts to an equivalence on the full subcategories of smooth objects holds because of the fact that the nonlinear Cartier descent has this property. This completes the whole proof. 
\end{proof}
\begin{remark}\label{extend OV}
Argued as \cite[\S3]{LSZ15}, one sees that the equivalence in above theorem generalizes the equivalence constructed by Ogus-Vologodsky \cite[Theorem 2.11 (ii)]{OV} for linear objects with the same nilpotency condition. Therefore, we shall also call the functors \emph{Cartier/inverse Cartier transform}, and denote them by $C/C^{-1}$.  
\end{remark}

\section{Global nonlinear Hodge correspondence}\label{Global nonlinear Hodge correspondence}
In this section, we intend to relax Assumption \ref{assumption of scheme and frobenius lifting} to the $W_2(k)$-liftability of $S$ only. Our technique is the exponential twisting of \cite{LSZ15} and later further developed in \cite{SSW} for principal $G$-bundles. \\

A morphism $X\to S$ is said to be \emph{locally projective}, if for every point $s\in S$, there is an Zariski open neighborhood $U_s$ of $s$ such that $X_{U_s}:=X\times_SU_s\to U_s$ is projective. In the following, we consider \emph{flat} and locally projective morphisms. For a locally noetherian $S$-scheme $T$, let $\mathfrak{Aut}_{X/S}(T)$ be the set of of automorphisms of $X_T=X\times_ST$ over $T$. This contravariant functor $T\mapsto \mathfrak{Aut}_{X/S}(T)$ is representable by a finite type group scheme $\Aut_{X/S}$ over $S$ (see \cite{N} \footnote{In \S6 loc. cit., the representability is stated only for projective flat morphisms. However, it also holds for locally projective flat morphisms, since one may glue the local representable group scheme into a global one.}), which is called the relative automorphism group of $f$.  It is separated but not necessarily smooth. To simply our discussion, we shall make the following assumption on $\Aut_{X/S}$.
\begin{assumption}\label{relative automorphism group scheme}
The connected component $\Aut^0_{X/S}$ of $\Aut_{X/S}$ is a principal $G$-bundle over $S$, where $G$ is a split reductive group over $k$.      
\end{assumption}
Under this assumption, $f_*T_{X/S}$ is locally free and there is a canonical isomorphism 
$$
\mathfrak{aut}_{X/S}\cong f_*T_{X/S}.
$$
where $\mathfrak{aut}_{X/S}$ is the Lie algebra of the relative automorphism group of $f$. Before carrying out the exponential twisting process, we shall digress into the theory of exponential map in characteristic $p$, as developed by Serre, Seitz, Sobaje and Balaji-Deligne-Parameswaran. The following definition is taken from \cite[\S2.1]{BDP}.   
\begin{definition}
Let $R=\cup_i R_i$ be the decomposition of the root system of $G$ into the union of its irreducible sub root systems. For each irreducible $R_i$, one fixes a set of positive roots $R_i^+$. The Coxeter number $h_i$ is defined to be $1+\mathrm{ht}(\alpha_i)$, where $\alpha_i$ is the highest root. The Coxeter number $h_{G}$ of $G$ is defined to be $\max_i\{h_i\}$.      
\end{definition}
Let $\mathfrak{g}$ be the Lie algebra of $G$. A basic fact is that any nilpotent element $\mathfrak n$ in $\mathfrak g$ is $p$-nilpotent, if $p\geq h_{G}$. Let $\mathfrak{g}_{nil}$ (resp. $G^u$) be the reduced subscheme of $\mathfrak g$ (resp. $G$) with points the nilpotent (resp. unipotent) elements. The following result is stated in \cite[\S2.2]{BDP} with a proof in \S6 loc. cit., expanding a sketch by Serre. 
\begin{theorem}[Serre, Balaji-Deligne-Parameswaran]\label{exponential}
Suppose $p>h_G$. Then there exists a unique $G$-equivariant isomorphism (as schemes)
$$
\exp: \mathfrak{g}_{nil}\to G^u,
$$
which restricts to a $B$-equivariant isomorphism of algebraic groups whose differential at the origin is the identity
$$
\exp: (\mathfrak{u},\circ)\cong U,
$$
where $U$ is the unipotent radical of a Borel subgroup $B$ of $G$, $\mathfrak u$ is its Lie algebra and the group law $\circ$ is given by the Campbell-Hausdorff formula. 
\end{theorem}
Suppose $p>h_G$. Let $\tilde S$ be a $W_2(k)$-lifting of $S$. Let $(g,\theta)$ be a nonlinear Higgs bundle \footnote{Here we abuse bit the notion of a nonlinear Higgs bundle (see Definition \ref{nonlinear flat bundle and Higgs bundle}. The same abuse occurs in several other places in the article.), in the sense that the structural morphism is only required to be flat.} over $S'$, where $g: Y\to S'$ is a flat and locally projective morphism satisfying Assumption \ref{relative automorphism group scheme} (replace $S$ with $S'$), and $\theta$ is locally nilpotent. Take a finite open covering $S=\cup_{\alpha}S_{\alpha}$ such that the following properties hold: 
\begin{itemize}
    \item [(i)] Each $(S_{\alpha},F_{S_{\alpha}/k})$ is $W_2(k)$-liftable; 
    \item [(ii)] Each $S_{\alpha\beta}:=S_{\alpha}\cap S_{\beta}$ is affine;
    \item [(iii)] The connected component of the relative automorphism group of $g$ is trivializable over each $S'_{\alpha\beta}$.  
\end{itemize}
Choose and then fix a lifting $\tilde F_{\alpha}$ of $F_{S_{\alpha}/k}$ over $W_2(k)$. Set 
$$
\zeta_{\alpha}:=\frac{d\tilde F_{\alpha}}{[p]}: T_{S_{\alpha}}\to F^*_{S_{\alpha}/k}T_{S'_{\alpha}}.
$$
Then Deligne-Illusie's lemma (see e.g. \cite[Lemma 2.1]{LSZ15}) provides an element 
$$h_{\alpha\beta}\in F_{S'/k}^*T_{S'}(S'_{\alpha\beta})$$ such that 
\begin{itemize}
    \item [(i)] $dh_{\alpha\beta}=\zeta_{\alpha}-\zeta_{\beta}$; 
    \item [(ii)] $h_{\alpha\beta}+h_{\beta\gamma}=h_{\alpha\gamma}$ over $S_{\alpha\beta\gamma}:=S_{\alpha}\cap S_{\beta}\cap S_{\gamma}$. 
\end{itemize}
Let $g^{'}: Y^{'}\to S$ be the base change of $g$ via $F_{S'/k}: S\to S'$. We are going to twist $g^{'}$ as follows: Consider the composite (where $\theta_{\alpha\beta}$ is the restriction of $\theta$ to $S_{\alpha\beta}$)
$$
\sO_{Y'_{\alpha\beta}}=g'^*\sO_{S'_{\alpha\beta}}\stackrel{g'^*h_{\alpha\beta}}{\to} g'^*F^*_{S'_{\alpha\beta}/k}T_{S'_{\alpha\beta}}=\pi^*g^*T_{S'_{\alpha\beta}}\stackrel{\pi^*\theta_{\alpha\beta}}{\to} T_{Y'_{\alpha\beta}/S_{\alpha\beta}},
$$
which gives an element in $H^0(S_{\alpha\beta},g'_*T_{Y'/S})\cong T_{Y'/S}(Y'_{\alpha\beta})$. By the functorial property of $\mathfrak{Aut}$-functor, the representing $S$-group scheme $\Aut_{Y'/S}$ of $\mathfrak{Aut}_{Y'/S}$ is canonically isomorphic to $F_{S'/k}^*\Aut_{Y/S'}$. So is its connected component. By Assumption \ref{relative automorphism group scheme}, $\Aut^0_{Y/S'}$ is a principal $G$-bundle, so is $\Aut^0_{Y'/S}$. Hence the above discussion shows that
$$
\pi^{*}\theta_{\alpha\beta}\circ g'^*h_{\alpha\beta}\in \mathfrak{g}(\sO_{S_{\alpha\beta}}).
$$
As $\theta$ is locally nilpotent, it actually lies in $\mathfrak{g}_{nil}(\sO_{S_{\alpha\beta}})$. Moreover, because of the integrability of $\theta$, it follows from Engel's theorem that the element lies in $\mathfrak{u}(\sO_{S_{\alpha\beta}})$ for some $\mathfrak{u}$ (which may not be unique). Applying Theorem \ref{exponential} to this element, we obtain an element 
$$
g_{\alpha\beta}:=\exp (-\pi^*\theta_{\alpha\beta}\circ g'^*h_{\alpha\beta})\in G^u(\sO_{S_{\alpha\beta}})\subset G(\sO_{S_{\alpha\beta}}),
$$
which is naturally regarded as a local section of $\Aut^0_{Y'/S}\to S$ over $S_{\alpha\beta}$. On the other hand, applying Theorem \ref{local nonlinear Hodge correspondence} to $(g_{\alpha}: g^{-1}(S'_{\alpha})\to S'_{\alpha}, \theta_{\alpha}:=\theta|_{S'_{\alpha}})$, we obtain a local nonlinear flat bundle $(g'_{\alpha}, \nabla_{\alpha})$, where 
$$
\nabla_{\alpha}:=\nabla_{can}-\pi^*\theta_{\alpha}\circ g^{'*}_{\alpha}\zeta_{\alpha}.
$$
Note that $\theta$ is actually locally $p$-nilpotent. 
\begin{lemma}\label{gluing lemma}
Regard $\{g_{\alpha\beta}\}$ as an $S_{\alpha\beta}$-automorphism of $Y'_{\alpha\beta}$. Then the collection $\{g_{\alpha\beta}\}$ of local automorphisms satisfies the following properties:
\begin{itemize}
    \item [(i)] $dg_{\alpha\beta}\circ \nabla_{\alpha}|_{S_{\alpha\beta}}=\nabla_{\beta}|_{S_{\alpha\beta}}$ over $S_{\alpha\beta}$;
    \item [(ii)] $g_{\alpha\beta}\cdot g_{\beta\gamma}=g_{\alpha\gamma}$ over $S_{\alpha\beta\gamma}$.
\end{itemize}
\end{lemma}
\begin{proof}
Let us verify (ii) first. As said above, the image of $\theta$ over $S_{\alpha\beta\gamma}$ is contained in a commutative subalgebra of $\mathfrak{u}(\sO_{S_{\alpha\beta\gamma}})$. By Theorem \ref{exponential}, the exponential map over $\mathfrak{u}$ is a group homomorphism. Hence we have
\begin{eqnarray*}
 \exp(-\pi^*\theta\circ g'^*h_{\alpha\beta})\cdot \exp(-\pi^*\theta\circ g'^*h_{\beta\gamma})&=&\exp(-\pi^*\theta\circ g'^*h_{\alpha\beta}-\pi^*\theta\circ g'^*h_{\beta\gamma})\\
 &=&\exp(-\pi^*\theta\circ (g'^*h_{\alpha\beta}+g'^*h_{\beta\gamma}))\\
 &=&\exp(-\pi^*\theta\circ g'^*h_{\alpha\gamma}).
\end{eqnarray*}
This shows (ii). The equality in (i) boils down to the following two equalities
$$
dg_{\alpha\beta}\circ \nabla_{can}=\nabla_{can};\quad dg_{\alpha\beta}\circ \pi^*\theta\circ g^{'*}_{\alpha}\zeta_{\alpha}=\pi^*\theta\circ g^{'*}_{\beta}\zeta_{\beta}.
$$
As $g_{\alpha\beta}$ is an $S_{\alpha\beta}$-automorphism of $Y'_{\alpha\beta}$, the image of $dg_{\alpha\beta}\circ \nabla_{can}$ defines a transversal foliation of $g'_{\alpha\beta}$. Since $dg_{\alpha\beta}$ preserves Lie bracket, the foliation has also vanishing $p$-curvature. In other words, it is just the canonical connection $\nabla_{can}$. This shows the first equality. Note that $\pi^*\theta\circ g^{'*}_{\alpha}\zeta_\alpha|_{S_{\alpha\beta}}\in \mathfrak{u}(\sO_{S_{\alpha\beta}})$. Thus by Theorem \ref{exponential}, it follows that
\begin{eqnarray*}
 dg_{\alpha\beta}(\pi^*\theta\circ g^{'*}_{\alpha}\zeta_{\alpha})&=& d_{e}\exp (-\pi^*\theta\circ g'^*h_{\alpha\beta})\cdot \exp (\pi^*\theta\circ g^{'*}_{\alpha}\zeta_{\alpha}) \\
 &=&d_{e}\exp (-\pi^*\theta\circ g'^*h_{\alpha\beta}+\pi^*\theta\circ g^{'*}_{\alpha}\zeta_{\alpha})\\
 &=&-\pi^*\theta\circ g^{'*}h_{\alpha\beta}+\pi^*\theta\circ g^{'*}_{\alpha}\zeta_{\alpha}\\
 &=&\pi^*\theta\circ g^{'*}(-h_{\alpha\beta}+\zeta_{\alpha})\\
 &=&\pi^*\theta\circ g^{'*}_{\beta}\zeta_{\beta}.
\end{eqnarray*}
This completes the proof. 
\end{proof}
Because of (ii) in the lemma, we may glue these local morphisms $\{g'_{\alpha}: Y'_{\alpha}\to S_{\alpha}\}_{\alpha}$ via $\{g_{\alpha\beta}\}_{\alpha\beta}$, into a new morphism $f: X\to S$-this is the twist of the original morphism $g'$. Furthermore, because of (i) of the lemma, we also obtain a flat connection $\nabla$ on $f$. By construction, the morphism $f$ is smooth, locally projective, and satisfies Assumption \ref{relative automorphism group scheme} with the same $G$ for $g$. It is routine to check that the pair $(f,\nabla)$, up to isomorphism, is independent of the choices of open coverings on $S$ and local Frobenius liftings. \\

The above process is reversible. Indeed, given a nonlinear flat bundle $(f,\nabla)$, where $f$ is flat, locally projective and satisfies Assumption \ref{relative automorphism group scheme}, we apply Theorem \ref{local nonlinear Hodge correspondence} (and also Remark \ref{local correspondence preserves flatness} for the issue on flatness of morphisms) in the converse direction to obtain the collection of local nonlinear Higgs bundles, and then replace $\pi^*\theta$ in the above discussion with $\psi_{\nabla}$ to obtain the gluing datum. Gluing them, we obtain a global nonlinear Higgs bundle $(g,\theta)$, where $g$ is flat, locally projective and satisfies Assumption \ref{relative automorphism group scheme}. \\

To state the main result of this section, we form two categories: Fix $G$ as above. Let $\NMIC_{nil}^{G}(S)$ (resp. $\NHIG_{nilp}^{G}(S')$) be the full subcategory of $\NMIC_{nil}^G(S)$ (resp. $\NHIG_{nilp}(S')$) consisting of pairs $(f,\nabla)$ (resp. $(g,\theta)$), where $f$ (resp. $g$) is flat, locally projective and the connected component of its relative automorphism group is a principal $G$-bundle.

\begin{theorem}\label{global nonlinear Hodge correspondence}
Assume that $S$ is $W_2(k)$-liftable. Let $G$ be a split reductive group over $k$ with $h_{G}<p$. Then a $W_2(k)$-lifting of $S$ induces an equivalance of categories between $\NMIC^G_{nil}(S)$ and $\NHIG^G_{nil}(S')$. The equivalence is obtained by gluing the local equivalences in Theorem \ref{local nonlinear Hodge correspondence} via exponential twisting. 
\end{theorem}
\begin{remark}
Vector bundles over $S$ are not locally projective, rather quasi-projective. Therefore, we can not say the above result generalize the corresponding results for vector bundles \cite{OV}, \cite{LSZ15} or principal $G$-bundles \cite{SSW}. But it seems possible to extend the above discussions to locally quasi-projective morphisms, to which we choose to study on a later occasion.   
\end{remark}

\section{Nonlinear Fontaine modules}\label{nonlinear fontaine module section}
Fontaine modules (also known as Fontaine-Laffaille-Faltings modules) are analogues of variations of Hodge structure in positive/mixed characteristic. We refer our reader to \cite[II. d)]{F} or \cite[\S2]{LSZ19} for the notion. In this section, we would like to provide a nonlinear generalization. \\

Let $\mathbf S$ be a smooth $W(k)$-scheme, equipped with a Frobenius lifting $ \mathbf{F}:\hat{\mathbf{S}}\to \hat{\mathbf{S}}$ on its $p$-adic completion. For $n\in \N$, set $S_n=\mathbf{S}\times_{W(k)}W(k)/p^{n+1}$ and $S=S_0$. Let $F_n: S_n\to S_n$ be the reduction of $\mathbf{F}$.  
\begin{definition}\label{nonlinear fontaine module}
Fix $n\geq 0$. A nonlinear Fontaine module over $S_n$ is a quadruple $(f,\nabla,Fil,\varphi)$, where 
\begin{itemize}
    \item [(i)] $f: X\to S_n$ is a flat morphism over $W_n:=W_n(k)$;
    \item [(ii)] $\nabla$ is an integrable connection on $f$;
    \item [(iii)] $Fil$ is a filtration structure on $f$ given as follows: $${\tilde f}: \mathscr{X} \to \A^1_{S_n}:=\Spec(W_n[t])\times_{W_n}S_n$$ is a flat $W_n$-morphism sitting in the following Cartesian diagram
    \begin{equation*}
		\begin{gathered}
		\xymatrix{
			X \ar[r]^{} \ar[d]_{f} & \mathscr{X} \ar[d]^{\tilde f} \\
			S_n \ar[r]^{t=1} & \A^1_{S_n},
		}
		\end{gathered}  
		\end{equation*}
    together with a $\Gm$-action $\Gm\times \mathscr{X}\to \mathscr{X}$ which covers the standard $\Gm$-action on $\A^1_{S_n}$ (trivial on the factor $S_n$). For a $W_n$-point $a$ of $\A^1_{W_n}$, denote $f_a: \mathscr{X}_t\to S_n$ for the fiber of $\tilde f$ over $a$;
    \item [(iv)]  $\varphi$ is a Frobenius structure on $(f,\nabla,Fil)$ given as follows: $\varphi: X\to \mathscr{X}_p$ is an $F_n$-linear morphism, that is, $\varphi$ sits in the following commutative diagram 
    \begin{equation*}
		\begin{gathered}
		\xymatrix{
			X \ar[r]^{\varphi} \ar[d]_{f} & \mathscr{X}_p \ar[d]^{f_p} \\
			S_n \ar[r]^{F_n} & S_n,
		}
		\end{gathered}  
		\end{equation*}
         such that $\psi: X \to F_{n}^*\mathscr{X}_p$ is an isomorphism, where $\varphi=\pi\circ \psi$ and $\pi: F_{n}^*\mathscr{X}_p\to \mathscr{X}_p$ is the natural projection,
\end{itemize}
satisfying the following two conditions:
\begin{itemize}
    \item [(a)] Griffiths transversality on $Fil$. $\nabla: f^*T_{S/k}\to T_{X/S}$ is extendable to a $\Gm$-equivariant morphism
    $$
    \tilde \nabla: \tilde f^*T_{\A^1_{S_n}/\A^1_{W_n}}\to T_{\mathscr{X}/\A^1_{W_n}}
    $$
    whose composite with the natural projection $T_{\mathscr{X}/\A^1_{W_n}}\to \tilde f^*T_{\A^1_{S_n}/\A^1_{W_n}}$ equals $t\cdot id$. Such an extension, if exists, must be unique. Set $\theta=\tilde \nabla|_{\mathscr{X}_p}$;
\item [(b)] Parallelism on $\varphi$. The following diagram is commutative:
\begin{equation*}
		\begin{gathered}
		\xymatrix{
			f_1^*T_{S_n} \ar[r]^{\nabla} \ar[d]_{f_1^*\frac{d F_{n+1}}{[p]}} & T_{X} \ar[d]^{d\varphi} \\
			f_1^*F_n^*T_{S_n}=\varphi^*f_p^*T_{S_n} \ar[r]^{\quad \quad \quad \varphi^*\theta} & \varphi^*T_{\mathscr{X}_p}.
		}
		\end{gathered}  
		\end{equation*}
 
\end{itemize}
For two nonlinear Fontaine modules $(f_i,\nabla_i,Fil_i,\varphi_i), i=1,2$ over $S_n$, a morphism from $(f_1,\nabla_1,Fil_1,\varphi_1)$ to $(f_2,\nabla_2,Fil_2,\varphi_2)$ is a morphism $(f_1,\nabla_1)\to (f_2,\nabla_2)$ of nonlinear flat bundles over $S_n$ which respects the filtration structures and the Frobenius structures. 
\end{definition}
Though we have defined the notion over an arbitrary truncated Witt ring, we choose to focus on the $n=0$ case in what follows. Note if we stick to the characteristic $p$ situation, it is unnecessary to assume the existence of $\mathbf{S}$ and $\mathbf{F}$. We assume simply $S$ to be local and we fix a $W_2$-lifting $(\tilde S,\tilde F_S)$ of the pair $(S,F_S)$. \\

It is worthwhile to explain where the key ingredients of the above notion come from. The filtration structure and the Griffiths transversality condition are abstracted from the Hodge filtration on nonabelian cohomology \cite{Si97}. Let us first recall the following well-known construction: Let $\sV$ be a locally free $\sO_S$-module of finite rank, and $Fil^{\bullet}$ a finite decreasing and locally split (but not necessarily effective) filtration on $\sV$. We may loosely call such a filtration a Hodge filtration. The Rees construction (see e.g. \cite[\S5]{Si97}) associates the pair $(\sV,Fil^{\bullet})$ a locally free sheaf
$\xi(\sV,Fil^{\bullet})$ over $\A^1_S$, equipped with a $\G_m$-structure (where the $
\G_m$-action on $\sO_{\A^1_S}=\sO_S[t]$ comes from the standard $\G_m$-action on $\A^1$), i.e. $\xi(\sV,Fil^{\bullet})$ is a $\G_m$-equivariant locally free sheaf of finite rank over $\A^1_S$. In concrete term, $\xi(\sV,Fil^{\bullet})$ is the $\sO_S[t]$-submodule 
$$
\sum_iFil^i\otimes t^{-i}\subset \sV\otimes _{\sO_S}\sO_{S}[t,t^{-1}],
$$
and the $\G_m$-action is induced from the natural one on $\sV\otimes _{\sO_S}\sO_{S}[t,t^{-1}]$. Taking the fiber of $\xi(\sV,Fil^{\bullet})$ at $t=1$, we get back $\sV$ as well as $Fil^{\bullet}$: The $\G_m$-action gives a trivialization of $\xi(\sV,Fil^{\bullet})$ over $\G_m\subset \A^1$. An element $v\in Fil^i\subset \sV$ if and only if $t^{-i}v\in \xi(\sV,Fil^{\bullet})$. Using the same proof as Simpson gave for $S=\Spec\ k$ in loc. cit., the $\xi$-construction and taking the fiber at $t=1$ establish a natural bijection between the set of pairs $(\sV,Fil^{\bullet})$ and the set of $\G_m$-equivariant locally free $\sO_{\A^1_S}$-modules of finite rank. The following lemma should be known to experts.
\begin{lemma}\label{taking grading}
Let $\sV$ be as above. Let $Fil^{\bullet}$ be a finite decreasing filtration on $\sV$. Then it is locally split if and only if $\xi(\sV,Fil^{\bullet})$ is locally free. 
\end{lemma}
\begin{proof}
One direction is given in \cite[page 237]{Si97}: Suppose $Fil^{\bullet}$ is locally split. Then locally on $S$, there exists a so-called adapted basis $\{v_i\}$ of $\sV$ such that $v_i\in Fil^{p(i)}-Fil^{p(i)+1}$. Then $\xi(\sV,Fil^{\bullet})$ is locally the free $\sO_S[t]$-module with basis $\{t^{-p(i)}v_i\}$. Let us show the converse: Suppose $\xi(\sV,Fil^{\bullet})$ is locally free. Then its fiber at $t=0$ is also locally free. Our observation is that there is always a natural isomorphism 
$$
Gr_{Fil^{\bullet}}\sV\cong \xi(\sV,Fil^{\bullet})|_{t=0}.
$$
Without loss of generality, we may assume $S=\Spec\ A$ (which is noetherian), and $\sV$ as well as $Fil^i\subset \sV$ for each $i$ is the sheafificiation of finitely generated modules $M^i\subset M$. Then $\xi(\sV,Fil^{\bullet})|_{t=0}$ corresponds to the quotient module $\tilde M/t\tilde M$, where
$$
\tilde M=\sum_i t^{-i}M^i\subset M[t,t^{-1}]=\oplus_i M\{t^i\}.
$$
Consider the morphism of $A/tA$-modules
$$
\oplus_i M^i/M^{i+1}\to \tilde M/t\tilde M,\quad (\bar m_i)\mapsto \overline{\sum_i t^{-i}m_i},
$$
where $m_i\in M_i$ is a representative of $\bar m_i\in M^i/M^{i+1}$. It is clearly surjective. Suppose $\overline{\sum_i t^{-i}m_i}=0$. Then we have an equality of elements in $M[t,t^{-1}]$
$$
\sum_i t^{-i}m_i=\sum_i t^{-i+1}n_{i},
$$
where $n_i\in M^i$. Thus $m_i=n_{i+1}\in M^{i+1}$ and hence $(\bar m_i)=0$. Simple induction on $i$ shows that $Fil^i$ is locally a direct summand of $Fil^{i-1}$. The lemma is proved. 
\end{proof} 
For an integrable connection $(\sV,\nabla)$, $Fil^{\bullet}$ on $\sV$ satisfies the Griffiths transversality if and only if $\nabla$ extends in $\G_m$-equivariant way a $t$-connection $\tilde \nabla$ over $\A^1_S/\A^1$ (see \cite[Lemma 7.1]{Si97}). Such a triple $(\sV,\nabla,Fil^{\bullet})$ may be called a de Rham bundle. 
\begin{proposition}\label{filtration structure and griffiths transversality}
For a de Rham bundle $(\sV,\nabla, Fil^{\bullet})$ over $S$, the Rees construction gives rise to a linear connection $\nabla_{\pi}$ on a smooth $S$-variety $\pi: V\to S$ together with a filtration structure $Fil_{\pi}$ on $\pi$ which obeys the Griffiths transversality.  
\end{proposition}
\begin{proof}
This is a direct consequence of the previous discussions. Indeed, applying Proposition \ref{flat/higgs bundles are examples of foliated/higgs varieties} to $(\xi(\sV,Fil^{\bullet}),\tilde \nabla)$ yields a $t$-connection $\tilde \nabla_{\tilde \pi}$ on 
$$
\tilde \pi: \tilde V=\mathbf{Spec}S(\xi(\sV,Fil^{\bullet}))\to \A^1_S,
$$
together with a natural $\G_m$-action on $\tilde V$ coming from $\G_m$-structure on $\xi(\sV,Fil^{\bullet})$. Note also that the pullback of $\tilde \pi$ along $S\stackrel{t=1}{\to}\A^1_S$ is by definition equal to
$$
\mathbf{Spec}(S(\xi(\sV,Fil^{\bullet}))|_{t=1})=\mathbf{Spec}S(\xi(\sV,Fil^{\bullet})|_{t=1})=\mathbf{Spec}S(\sV)=V\stackrel{\pi}{\to} S.
$$
By going through the construction of linear connections in Proposition \ref{flat/higgs bundles are examples of foliated/higgs varieties}, one verifies furthermore that $\tilde \nabla_{\tilde \pi}$ pulls back to $\nabla_{\pi}$ at $t=1$. We omit the detail of verification. 
\end{proof}
It is also clear from the proof of Lemma \ref{taking grading} that the fiber of $\tilde \pi$ at $t=0$ coincides with the associated linear Higgs field to the graded Higgs sheaf $Gr_{Fil^{\cdot}}(\sV,\nabla)$. Because of this fact, we shall call $(f_0,\theta)$ in general, the fiber of $\tilde f$ at $t=0$ in Definition \ref{nonlinear fontaine module}, the associated graded Higgs $S$-variety to $(f,\nabla,Fil)$. \\

The Frobenius structure and the parallel condition come from Faltings' formulation on the strongly $p$-divisible property and the parallel condition satisfied by relative Frobenius on cohomology (\cite{F}). In the remaining paragraphs, our discussion shall center around this point. Let us first introduce the following 
\begin{definition}\label{strict p-torsion nonlinear Fontaine module}
A geometric strict $p$-torsion nonlinear Fontaine module over $S$ is a nonlinear Fontaine module over $S$ such that $\theta$ is $p$-nilpotent, viz. $[p]\circ \theta=0$.    
\end{definition}
Recall from \cite[Definition 4.6]{OV} that a Fontaine module over $S$ is a quadruple $(\sV,\nabla,Fil^{\bullet},\psi)$ where $(\sV,\nabla,Fil^{\bullet})$ is a de Rham bundle over $S$, whose Hodge filtration is of level $\leq p-1$, and $\psi$ is an isomorphism  
$$
\psi: C^{-1}q^*Gr_{Fil^{\bullet}}(\sV,\nabla)\to (\sV,\nabla),
$$
where $q: S'=S\times_{k,\sigma}k\to S$ is the natural projection. 
\begin{proposition}\label{classical Fontaine modules}
A Fontaine module a la Ogus-Vologodsky is naturally a geometric strict $p$-torsion nonlinear Fontaine module. 
\end{proposition}
\begin{proof}
After Proposition \ref{filtration structure and griffiths transversality}, it remains to show a Frobenius structure on $(\pi: V\to S, \nabla_{\pi}, Fil_{\pi})$ which satisfies the paraellel condition. Recall that, over a local $S$, the inverse Cartier transform of Ogus-Vologodsky for nilpotent Higgs bundles $(\sE,\theta)$ of exponent $\leq p-1$ is simply given by
$$
(F_S^*\sE,\nabla_{can}-\frac{d\tilde F_S}{[p]}\circ F^*\theta)\footnote{This has been verified in \cite{LSZ15}. Unfortunately, there is a mistake about sign: The functor $C^{-1}_{exp}$ does not differ from $C^{-1}$ by a sign as stated in \cite[\S3]{LSZ15}. They actually simply coincide. This is because in the local construction of $C^{-1}_{exp}$ (\cite[\S2.2]{LSZ15}), the connection should read $\nabla_{can}-\frac{d\tilde F_S}{[p]}\circ F_S^*\theta$, instead of $\nabla_{can}+\frac{d\tilde F_S}{[p]}\circ F_S^*\theta$.}. 
$$
Then the inverse of the sheaf isomorphism $\psi$ gives rise to a bundle isomorphism $V\cong F_S^*E$, where $E$ is the geometric vector bundle associated to $\Gr_{Fil}\sV$. Let $\varphi: V\to E$ be the composite of the natural projection $F_S^*E\to E$ with the previous isomorphism. Then the condition that $\psi$ preserves the connection translates exactly into the parallel condition as required for $\varphi$.
\end{proof}
In \cite{Sh}, we shall study another example of geometric strict $p$-torsion nonlinear Fontaine modules which comes from nonabelian Hodge moduli spaces. In connection to the local nonlinear Hodge correspondence, the significance of the above notion lies in the following fact.
\begin{proposition}\label{nonlinear Fontaine moduli implies one-periodicity}
Let $(f,\nabla,Fil,\varphi)$ be a geometric strict $p$-torsion nonlinear Fontaine module over $S$. Then the Frobenius structure on $f$ induces an isomorphism of transversal foliated $S$-varieties: 
$$
C^{-1}q^*(f_0,-\theta)\cong (f,\nabla). 
$$
\end{proposition}
\begin{proof}
In a char $p$ ring, $p=0$. Hence $\psi$ is an $S$-isomorphism $X \to X_0^{(p)}:=F_S^*X_0$. As $\theta$ is $p$-nilpotent, $C^{-1}(f_0,-\theta)$ is given by
$$
(f_0^{(p)}, \nabla_{can}+\pi^*\theta\circ f_0^{(p)}\zeta),
$$
where $f_0^{(p)}: X_0^{(p)}\to S$ is the base change of $f_0$ via $F_S$. It suffices to show, under the isomorphism $\psi$, $\nabla$ is identified with $\nabla_{can}+\pi^*\theta\circ f_0^{(p)}\zeta$. But this is exactly what the parallel condition for $\varphi$ says: Through $\psi$, $X_1$ is identified with the Cartesian product $X_0\times_{S,F_S}S$. In this way, $\varphi$ is identified with the natural projection $\pi$, $f_0^{(p)}$ is identified with $f_1$. Note that $\nabla$ yields a splitting 
$$
T_{X_1}\cong f_1^*T_{S}\oplus T_{X_1/S}\cong f_1^*T_{S/k}\oplus \pi^*T_{X_0/S}.
$$
over which $d\pi$ acts as $0\oplus id$, that is  
$$
(a,b)\mapsto (0,b),\quad a\in f_1^*T_{S/k}, \ b\in \pi^*T_{X_0/S}.
$$
This is because for the component $a$, $d\pi$ acts as $dF_S$ which is zero. Now the commutative diagram in condition (b) of Definition \ref{nonlinear fontaine module} says that $\nabla$ acts as $\nabla_{can}+\pi^*\theta\circ f_1^*\zeta$ as wanted.
\end{proof}
\begin{remark}
It is an interesting question to ask whether one obtains a generalization of \cite[Theorem 2.6]{F} for geometric nonlinear Fontaine modules.     
\end{remark}

To conclude this section, we want to globalize this notion, using the construction in \S\ref{Global nonlinear Hodge correspondence}. Note that the essence of that construction is via the truncated exponential series, one obtains a \emph{canonical} isomorphism of functors between two local (inverse) Cartier transforms. Let $G$ be a split reductive group over $k$ with $h(G)<p$. 
\begin{definition}\label{global geometric nonlinear fontaine module}
Assume that $S$ is $W_2(k)$-liftable and fix such a lifting $\tilde S$. A geometric strict $p$-torsion nonlinear Fontaine module with symmetric group $G$ over $S$ (with respect to $\tilde S$) is a quadruple $(f,\nabla,Fil,\varphi)$, where $(f,\nabla)$ is an object in the category $\NMIC_{nil}^G(S)$, $Fil$ is a filtration structure on $f_1$ as defined in Definition \ref{nonlinear fontaine module} (iii) which satisfies Griffiths transversality as given in Definition \ref{nonlinear fontaine module} (a), $\varphi$ is a Frobenius structure on $(f,Fil)$ which is given as follows: For any choice $\tilde F_{\alpha}$ of local Frobenius liftings over an open subscheme $\tilde S_{\alpha}\subset \tilde S$, a Frobenius structure $\varphi_{(\tilde S_{\alpha},\tilde F_{\alpha})}$ is put on $(f,\nabla,Fil)|_{S_{\alpha}}$, called the evaluation of $\varphi$ over $(\tilde S_{\alpha},\tilde F_{\alpha})$, such that the following holds
\begin{itemize}
    \item [(i)] The quadruple $(f|_{S_{\alpha}},\nabla|_{S_{\alpha}},Fil|_{S_{\alpha}},\varphi_{(\tilde S_{\alpha},\tilde F_{\alpha})})$ is a geometric strict $p$-torsion nonlinear Fontaine module over $S_{\alpha}$ with respect to the Frobenius lifting $\tilde F_{S_{\alpha}}$.  
    \item [(ii)] Any two evaluations are related through the Taylor formula: Over $S_{\alpha\beta}$, it holds that
    $$
   \psi_{(\tilde S_{\alpha},\tilde F_{\alpha})}=\exp(\pi^*\theta_{\alpha\beta}\circ f_0^{(p)*}h_{\alpha\beta})\circ \psi_{(\tilde S_{\beta},\tilde F_{\beta})}.
    $$ 
\end{itemize}
\end{definition}
It is now a tautology to extend Proposition \ref{nonlinear Fontaine moduli implies one-periodicity} to a global $S$.
\begin{theorem}\label{global one-periodicity}
Assume $S$ is $W_2(k)$-liftable. Then for any geometric strict $p$-torsion nonlinear Fontaine module $(f,\nabla,Fil,\varphi)$ with symmetric group $G$ over $S$ with respect to some $W_2(k)$-lifting of $S$, it holds that
$$
C^{-1}q^*(f_0,-\theta)\cong (f,\nabla),
$$
where $q: S'\to S$ is the natural projection.
\end{theorem}
\begin{proof}
Because of the Taylor formula in Definition \ref{global geometric nonlinear fontaine module} (ii),    
the set of local isomorphisms 
$$
C^{-1}q^*(f_0,-\theta)|_{S_{\alpha}}\cong (f,\nabla)|_{S_{\alpha}},
$$
induced by $\{\psi_{(\tilde S_{\alpha},\tilde F_{\alpha})}\}$s, glue altogether into a global isomorphism
$$
C^{-1}q^*(f_0,-\theta)\cong (f,\nabla).
$$
\end{proof}

\section{Appendix: Nonlinear Cartier descent and Ekedahl's correspondence, by Mao Sheng and Jinxing Xu}\label{Appendix}
In this appendix, we shall extend the nonlinear Cartier descent theorem to arbitrary transversally foliated $S$-varieties. Moreover, we shall give a new proof of the Ekedahl's correspondence using the nonlinear Cartier descent. 

\subsection{Nonlinear Cartier descent}\label{extension of nonlinear Cartier descent}
Assume that $k$ is a ring of characteristic $p$, and let $S$ be a smooth $k$-scheme. Set $S' \coloneqq S \times_{k,\sigma} k$, the base change of $S$ along the Frobenius homomorphism $\sigma \colon k \to k$, and let $F = F_{S/k} \colon S \to S'$ denote the relative Frobenius morphism.\\

Let $f' \colon X' \to S'$ be a $k$-morphism, and consider the base change 
\[f \colon X \coloneqq X' \times_{S'} S \to S.\]
 The Cartesian diagram induces a natural isomorphism
\[
   \Omega_{X/S'} \;\stackrel{\sim}{\longrightarrow}\; f^* \Omega_{S/S'} \;\oplus\; \pi^* \Omega_{X'/S'},
\]
where $\pi \colon X \to X'$ is the projection. Since both $S \to S'$ and $X \to S'$ induce zero homomorphisms on differentials, we have natural identifications $\Omega_{S/S'} = \Omega_{S/k}$ and $\Omega_{X/S'} = \Omega_{X/k}$.
It follows that there is an isomorphism
\begin{align}\label{eq:isomorphism of diffs}
    \Omega_{X/k} \;\stackrel{\sim}{\longrightarrow}\; f^* \Omega_{S/k} \;\oplus\; \pi^* \Omega_{X'/S'},
\end{align}
where $\pi \colon X \to X'$ is the projection. Since $\Omega_{S/k}$ is locally free, we have
\[
   \mathcal{H}om_{\mathcal{O}_X}(f^*\Omega_{S/k}, \mathcal{O}_X) \;\simeq\; f^* \mathcal{H}om_{\mathcal{O}_S}(\Omega_{S/k}, \mathcal{O}_S),
\]
and because the Frobenius morphism $S \to S'$ is flat,
\[
   \mathcal{H}om_{\mathcal{O}_X}(\pi^* \Omega_{X'/S'}, \mathcal{O}_X) \;\simeq\; \pi^* \mathcal{H}om_{\mathcal{O}_{X'}}(\Omega_{X'/S'}, \mathcal{O}_{X'}).
\]
Applying $\mathcal{H}om_{\mathcal{O}_X}(-,\mathcal{O}_X)$ to \eqref{eq:isomorphism of diffs} then yields
\begin{align}\label{eq:appendix splitting of tangent bundles}
       T_{X/k} \;\stackrel{\sim}{\longrightarrow}\; f^* T_{S/k} \;\oplus\; \pi^* T_{X'/S'}.
\end{align}
This splitting defines an $\mathcal{O}_X$-linear morphism
\[
   \nabla_{can} \colon f^* T_{S/k} \longrightarrow T_{X/k}.
\]
A straightforward verification shows that $\nabla_{can}$ is a transversal foliation on $f$ with vanishing $p$-curvature. Moreover, it coincides with the one defined in Lemma \ref{factorization lemma} when $k$ is a perfect field and $f'$ is smooth. Since $\pi \colon X \to X'$ is faithfully flat and induces a homeomorphism on the underlying topological spaces, we may identify $\mathcal{O}_{X'}$ with a subsheaf of $\mathcal{O}_X$. It then follows from the splitting \eqref{eq:isomorphism of diffs} that
\begin{align}\label{eq:appendix horizontal sections of canonical connection}
  \mathcal{O}_{X'} \;=\; \mathcal{O}_X^{\nabla_{can}} \;\coloneqq\; \{\, s \in \mathcal{O}_X \mid \nabla_{can}(s) = 0 \,\}.
\end{align}

Conversely, we have the following version of the nonlinear Cartier descent theorem, which removes the smoothness assumption on $X/S$ in Theorem \ref{equivalence in the case of vanishing curvature}.

\begin{theorem}\label{thm:nonlinear Cartier descent in appendix}
Keep $k$, $S$, and $S'$ as above. Let $f \colon X \to S$ be a quasi-compact $k$-morphism. If $\nabla \colon f^* T_{S/k} \to T_{X/k}$ is a transversal foliation with vanishing $p$-curvature, then there exists an $S'$-scheme $X'$ and an $S$-isomorphism 
\[
   X \;\stackrel{\sim}{\longrightarrow}\; X' \times_{S'} S
\]
such that $\nabla$ corresponds to $\nabla_{can}$ under this isomorphism. Moreover, the assignment $(X,\nabla) \mapsto X'$ defines an equivalence of categories between $\NMIC_0(S)$ and $\NHIG_0(S')$.
\end{theorem}

\begin{proof}
By \eqref{eq:appendix horizontal sections of canonical connection}, it suffices to show the existence of $X'$. Since $T_{S/k}$ is locally free and $T_{X/k} =  \operatorname{Der}_k(\mathcal{O}_X) \subset \mathcal{H}om_k(\mathcal{O}_X, \mathcal{O}_X)$, we may rewrite $\nabla$ as
\[
   \nabla_1 \colon \mathcal{O}_X \to f^* \Omega_{S/k},
\]
and, using the projection formula, further as a $k$-linear homomorphism
\[
   \nabla_2 \colon f_* \mathcal{O}_X \to f_* \mathcal{O}_X \otimes_{\mathcal{O}_S} \Omega_{S/k}.
\]
The assumption that $\nabla$ is a transversal foliation implies that $\nabla_2$ defines a $k$-linear connection on the $\mathcal{O}_S$-module sheaf $f_* \mathcal{O}_X$. Moreover, since $\nabla$ has vanishing $p$-curvature, so does $\nabla_2$. As $f$ is quasi-compact, the $\mathcal{O}_S$-module $\mathcal{E} \coloneqq f_* \mathcal{O}_X$ is quasi-coherent.  Applying the classical Cartier descent to $(\mathcal{E}, \nabla_2)$ yields an isomorphism of $\mathcal{O}_S$-modules
\[
   \varphi \colon F^*(\mathcal{E}^{\nabla_2}) \;\stackrel{\sim}{\longrightarrow}\; \mathcal{E}.
\]
It is straightforward to see that $\mathcal{E}^{\nabla_2} = f_* (\mathcal{O}_X^{\nabla})$, which implies that $\mathcal{E}^{\nabla_2}$ is a quasi-coherent $\mathcal{O}_{S'}$-subalgebra of $\mathcal{E}$, and that $\varphi$ is an isomorphism of $\mathcal{O}_S$-algebras. Defining the $S'$-scheme $X'$ by $X' \coloneqq \Spec ~ \mathcal{E}^{\nabla_2}$, the isomorphism $\varphi$ induces the desired $S$-isomorphism
\[
   X \;\stackrel{\sim}{\longrightarrow}\; X' \times_{S'} S.
\]
Finally, classical Cartier descent again implies that $\nabla$ corresponds to $\nabla_{can}$ under this isomorphism.
\end{proof}
\begin{remark}\label{remark on flatness in nonlinear Cartier descent}
Let $f': X'\to S'$ be the structural morphism. Since the morphism $F: S\to S'$ is faithfully flat, it follows from \cite[Proposition 17.7.4]{EGA IV} (resp. \cite[Corollaire 17.7.3 (ii)]{EGA IV} that in the nonlinear Cartier descent, $f$ is flat (resp. smooth) if and only if $f'$ is flat (resp. smooth). 
\end{remark}
Notice that the above proof does not use Ekedahl's correspondence Theorem \ref{Ekedahl's thm}. Hence it gives a new proof of Theorem \ref{equivalence in the case of vanishing curvature}, that works now for a more general base ring $k$ than a perfect field. 

\subsection{Ekedahl's correspondence via nonlinear Cartier descent}
Assume that $k$ is a ring of characteristic $p$, and let $X$ be a smooth $k$-scheme. Set $X' \coloneqq X \times_{k,\sigma} k$ for the base change along the Frobenius homomorphism $\sigma \colon k \to k$, and let $F = F_{X/k} \colon X \to X'$ be the relative Frobenius morphism. Recall a smooth foliation $\mathcal{F} \subset T_{X/k}$ is a coherent subsheaf such that $T_{X/k}/\mathcal{F}$ is locally free and $\mathcal{F}$ is closed under the Lie bracket. The foliation $\mathcal{F}$ is called $p$-closed if the composite 
\[
   \psi_{\mathcal{F}} \colon \mathcal{F} \to T_{X/k} \xrightarrow{[p]} T_{X/k} \to T_{X/k}/\mathcal{F}
\]
vanishes.

\begin{lemma}\label{lemma:appendix local representation of foliation}
   Let $\mathcal{F}$ be a coherent subsheaf of $T_{X/k}$ such that the quotient $T_{X/k}/\mathcal{F}$ is locally free. Then for any point $x \in X$, there exists an open neighborhood $U \subseteq X$ of $x$, a smooth $k$-scheme $S$, and a $k$-morphism $f \colon U \to S$ such that $f$ induces an isomorphism
   \[
      \mathcal{F}|_U \;\simeq \; f^*T_{S/k}.
   \]
\end{lemma}
\begin{proof}
Since $X/k$ is smooth, we can choose local coordinates $x_1, \dots, x_n$ of $X$ around $x$ such that 
\[
   \Omega_{X/k, x} \;=\; \mathcal{O}_{X,x}\, \mathrm{d}x_1 \oplus \cdots \oplus \mathcal{O}_{X,x}\, \mathrm{d}x_n.
\]
As $T_{X/k}/\mathcal{F}$ is locally free, the short exact sequence 
\[
   0 \;\longrightarrow\; \mathcal{F} \;\longrightarrow\; T_{X/k} \;\longrightarrow\; T_{X/k}/\mathcal{F} \;\longrightarrow\; 0
\]
splits locally. In particular, $\mathcal{F}$ is locally free, and a basis $e_1, \dots, e_r$ of the $\mathcal{O}_{X,x}$-module $\mathcal{F}_x$ can be extended to a basis of $T_{X/k,x}$. After possibly reordering the coordinates $x_i$, we may choose $e_1, \dots, e_r$ so that for any $1 \leq i,j \leq r$, the natural pairing satisfies
\begin{equation}\label{eq:appendix eixj pairings}
   \langle e_i, \mathrm{d}x_j \rangle \;=\;
   \begin{cases}
   1, & i = j, \\
   0, & i \neq j.
   \end{cases}
\end{equation}

Let $U$ be an affine open neighborhood of $x$ on which \eqref{eq:appendix eixj pairings} holds. Consider the polynomial ring $k[t_1, \dots, t_r]$. The $k$-algebra homomorphism
\[
   k[t_1, \dots, t_r] \;\longrightarrow\; \mathcal{O}_X(U), 
   \qquad t_i \longmapsto x_i
\]
induces a $k$-morphism 
\[
   f \colon U \longrightarrow S \coloneqq \Spec\ k[t_1, \dots, t_r].
\]
It is straightforward to verify that this morphism $f$ satisfies the desired property.
\end{proof}

We can now extend Ekedahl's correspondence (Theorem \ref{Ekedahl's thm}) to the case where the base is regular.
\begin{theorem}\label{thm:appendix Ekedahl cor}
  Assume $k$ is an integral regular ring of characteristic $p$ and $X$ is a smooth $k$-variety. Then there exists a one-to-one correspondence between the set of $p$-closed smooth foliations on $X$ and the set of (isomorphism classes of) inseparable morphisms of $k$-varieties
\[
X \stackrel{\pi}{\longrightarrow} Y \stackrel{\pi'}{\longrightarrow} X'
\]
with $F_{X/k}=\pi'\circ \pi$ and $Y$ non-singular.   
\end{theorem}
\begin{proof}
Since $X/k$ is smooth, the relative Frobenius $F_{X/k} \colon X \to X'$ is faithfully flat and induces a homeomorphism on the underlying topological spaces. Thus, we may regard $\mathcal{O}_{X'}$ as a subsheaf of $\mathcal{O}_X$. Given a $p$-closed smooth foliation $\mathcal{F}$ on $X$, we define $X \stackrel{\pi}{\longrightarrow} Y$ by
\[
\sO_Y=\Ann(\sF):=\{a\in \sO_X|\  Da=0, \ \textrm{for all}\ D\in \sF\}.
\]
It is clear that $\mathcal{O}_{X'} \subset \mathcal{O}_Y$, so there exists a morphism $\pi' \colon Y \to X'$ such that $F_{X/k} = \pi' \circ \pi$.  Conversely, given a chain of finite morphisms as above, we define
\[
\sF=\Ann(Y):=\{D\in T_{X/k} \mid Da=0 \ \text{for all}\ a\in \sO_Y\}.
\]
It then remains to verify the following claims:
\begin{compactenum}[\normalfont(i)]
\item Given $\mathcal{F}$, the corresponding $Y$ is smooth over $k$, and $\mathrm{Ann}(\mathrm{Ann}(\mathcal{F}))=\mathcal{F}$.
\item Given $X \stackrel{\pi}{\longrightarrow} Y \stackrel{\pi'}{\longrightarrow} X'$, we have $\mathrm{Ann}(\mathrm{Ann}(\mathcal{O}_Y))=\mathcal{O}_Y$.
\end{compactenum}

First, note that these claims can be verified locally. For (i), given $\mathcal{F}$, by Lemma \ref{lemma:appendix local representation of foliation}, we may assume there exists a $k$-morphism $f \colon X \to S$ to a smooth $k$-scheme $S$, such that $f$ induces an isomorphism $\mathcal{F} \simeq f^*T_{S/k}$. This isomorphism yields an $\mathcal{O}_X$-linear homomorphism
\[\nabla \colon f^*T_{S/k} \stackrel{\sim}{\longrightarrow} \mathcal{F}\hookrightarrow T_{X/k}.\]
By construction, $\nabla$ is a transversal foliation of $f$ with vanishing $p$-curvature. We may then apply Theorem \ref{thm:nonlinear Cartier descent in appendix} to obtain the following Cartesian diagram:
\[\begin{tikzcd}
    X && {X_1} \\
    & \square \\
    S && {S'}
    \arrow["{\pi_1}", from=1-1, to=1-3]
    \arrow["f"', from=1-1, to=3-1]
    \arrow[from=1-3, to=3-3]
    \arrow["{F_{S/k}}", from=3-1, to=3-3]
\end{tikzcd}\]
Moreover, $\nabla$ coincides with $\nabla_{can}$ induced by this Cartesian structure. From the splittings \eqref{eq:isomorphism of diffs} and \eqref{eq:appendix splitting of tangent bundles}, we deduce that $\mathcal{O}_Y \coloneqq \mathrm{Ann}(\mathcal{F})=\mathcal{O}_{X_1}$ as subsheaves of $\mathcal{O}_X$, and furthermore that $\mathrm{Ann}(\mathcal{O}_{X_1})=\operatorname{Im}\nabla_{can}=\mathcal{F}$. Hence we may identify $X \stackrel{\pi_1}{\rightarrow} X_1$ with $X \stackrel{\pi}{\rightarrow} Y$, and conclude that $\mathrm{Ann}(\mathrm{Ann}(\mathcal{F}))=\mathcal{F}$. Note that the sequence
\[0 \rightarrow f^*\Omega_{S/k} \rightarrow \Omega_{X/k} \rightarrow \Omega_{X/S} \rightarrow 0\]
is a split short exact sequence. By the smoothness of $X/k$, it follows that $X/S$ is also smooth. Since $F_{S/k}$ is faithfully flat, the smoothness of $X/S$ implies the smoothness of $X'/S'$, and therefore $X'/k$ is smooth as well. This completes the proof of (i).

Next, consider $X \stackrel{\pi}{\longrightarrow} Y \stackrel{\pi'}{\longrightarrow} X'$ with $Y$ regular. Without loss of generality, we may assume that $X=\operatorname{Spec} C$ and $Y=\operatorname{Spec} C'$ are affine. By Kunz’s conjecture, proved in \cite{KN}, the extension $C' \hookrightarrow C$ admits a $p$-basis \( c_1, \dots, c_n \in C \); that is, \( C \) is a free \( C' \)-module with basis
\[
\left\{ c_1^{i_1} \cdots c_n^{i_n} \,\middle|\, 0 \leq i_\alpha < p \text{ for all } 1 \leq \alpha\leq n \right\}.
\]
With this in place, we can construct the following Cartesian diagram:

\begin{equation}\notag
\begin{tikzcd}
    C && {C^{\prime}} \\
    & \square \\
    {B} && {B^{(p)}}
    \arrow[from=1-3, to=1-1]
    \arrow["\psi", from=3-1, to=1-1]
    \arrow["{\psi^{\prime}}", from=3-3, to=1-3]
    \arrow["{t_i^p \mapsfrom t_i}"', from=3-3, to=3-1]
\end{tikzcd}
\end{equation}
Here, \( B^{(p)} = B = k[t_1, \dots, t_n] \) is the polynomial algebra, and the morphism \( \psi' \) is the \( k \)-algebra homomorphism defined by \( t_i \mapsto c_i^p \), while \( \psi \) is the \( k \)-algebra homomorphism given by \( t_i \mapsto c_i \). Let $S=\operatorname{Spec} B$, we obtain the following Cartesian diagram:
\[\begin{tikzcd}
    X && {Y} \\
    & \square \\
    S && {S'}
    \arrow["{\pi}", from=1-1, to=1-3]
    \arrow[from=1-1, to=3-1]
    \arrow[from=1-3, to=3-3]
    \arrow["{F_{S/k}}", from=3-1, to=3-3]
\end{tikzcd}\]
Then, as in the proof of (i), it follows from the splittings \eqref{eq:isomorphism of diffs} and \eqref{eq:appendix splitting of tangent bundles} that $\mathrm{Ann}(\mathrm{Ann}(\mathcal{O}_Y ))=\mathcal{O}_Y$. This completes the proof of claim (ii).

\end{proof}

M.S., Yau Mathematical Science Center, Tsinghua University, Beijing, 100084, China \& Yanqi Lake Beijing Institute of Mathematical Sciences and Applications, Beijing, 101408, China\\
\noindent\small{Email: \texttt{msheng@tsinghua.edu.cn}}

\noindent  J.X., School of Mathematical Sciences, University of Science and Technology of China, Hefei, 230026, China\\
\noindent\small{Email: \texttt{xujx02@ustc.edu.cn}}

\end{document}